\documentclass{article}
    
    \usepackage[margin=0.7in]{geometry}
    \usepackage[parfill]{parskip}
    \usepackage[utf8]{inputenc}
    
    \usepackage{amsmath,amssymb,amsfonts,amsthm}

\usepackage{enumitem} 
\setlist[enumerate]{topsep=0pt,itemsep=-1ex,partopsep=1ex,parsep=1ex}

\usepackage{cleveref}

\usepackage{graphicx}
\usepackage{color}

\theoremstyle{plain}
\newtheorem{theo}{Theorem}[section]

\newtheorem{lemma}[theo]{Lemma}
\newtheorem{cor}[theo]{Corollary}

\theoremstyle{definition}
\newtheorem{defn}[theo]{Definition}
\newtheorem{rem}[theo]{Remark}

\newcommand{\mc}[1]{\mathcal{#1}}
\newcommand{\mb}[1]{\mathbb{#1}}

\newcommand{\nim}[1]{\noindent {\em #1}}
\newcommand{\brac}[1]{\left( #1 \right)}

\newcommand{\bsize}[1]{\left| #1 \right|}

\newcommand{\bgen}[1]{\left\langle #1 \right\rangle}

\newcommand{\bfl}[1]{\left\lfloor #1 \right\rfloor}
\newcommand{\bcl}[1]{\left\lceil #1 \right\rceil}

\newcommand{\sub}{\subseteq}

\newcommand{\sups}{\supseteq}

\newcommand{\sm}{\setminus}

\newcommand{\eps}{\varepsilon}
\newcommand{\es}{\emptyset}
\newcommand{\pl}{\partial}

\newcommand{\aA}{\alpha}

\newcommand{\tT}{\theta}
\newcommand{\sS}{\sigma}
\newcommand{\oO}{\omega}

\newcommand{\GG}{\Gamma}

\newcommand{\OO}{\Omega}

\newcommand{\Ups}{\Upsilon}

\newcommand{\hQ}{\hat{Q}}

\newcommand{\var}{\operatorname{Var}}

\title{A short proof of the existence of designs}
\author{Peter Keevash\thanks{Mathematical Institute, University of Oxford, UK. 
Supported by ERC Advanced Grant 883810.}}
\date{}

\begin{document}
\maketitle

\begin{abstract}
We give a new proof of the existence of designs,
which is much shorter and gives better bounds.
\end{abstract}

\section{Introduction}

The existence of designs was one of the oldest problems in combinatorics, 
studied by many 19th century mathematicians, including (perhaps in chronological order) 
Pl\"ucker, Kirkman, Sylvester, Woolhouse, Cayley and Steiner.
Steiner's name became attached to the fundamental object of interest,
namely a \emph{Steiner system} with parameters $(n,q,r)$, which is a set $S$ 
of $q$-subsets of an  $n$-set $X$ such that  every $r$-subset of $X$ 
is contained in exactly one element of $S$. There are some natural necessary
divisibility conditions (discussed below) on $n$ in terms of $q$ and $r$
for the existence of such a system $S$. The existence conjecture states that
these conditions suffice for $n>n_0(q,r)$ large enough in terms of $q$ and $r$.
There are now three quite different proofs of this conjecture, chronologically \cite{K, GKLO, DP};
we refer the reader to these papers for more history of the problem
and references to the large associated literature (the solution of the conjecture
has inspired several other dramatic breakthroughs in design theory).
The purpose of this paper is to give yet another proof, which has the advantages
of being much shorter and giving a reasonable bound for $n_0$.

To state the result, we adopt the formulation of hypergraph decompositions.
We denote the complete $r$-graph on a set $S$ by $\binom{S}{r}$:~this is 
a hypergraph with vertex set $S$ and edge set consisting of all $r$-subsets of $S$.
We also write $K^r_n = \tbinom{[n]}{r}$, where $[n]:=\{1,\dots,n\}$. 
We identify hypergraphs with their edge sets, so $|H|$ counts edges (we let $v_H$ count vertices).
A Steiner system with parameters $(n,q,r)$ is equivalent to a $K^r_q$-decomposition of $K^r_n$,
that is, a partition $D$ of (the edge set of) $K^r_n$ into copies of $K^r_q$. We can also think of $D$
as a perfect matching (partition of the vertex set into edges)
in the design hypergraph $K^r_q(K^r_n)$, where for an $r$-graph $G$ we write
$K^r_q(G)$ for the $\tbinom{q}{r}$-graph $H$ with $V(H)=G$ (vertices of $H$ are edges of $G$)
and edges consisting of all (edge sets of) copies of $K^r_q$ in $G$.
The necessary divisibility conditions mentioned above appear
in the integral relaxation, where for $\Phi \in \mb{Z}^{K^r_q(G)}$ we define 
$\pl \Phi \in \mb{Z}^G$ by $(\pl \Phi)_e = \sum \{ \Phi_Q: e \in Q \}$:~if $\pl \Phi = G' \sub G$
we call $\Phi$ an integral $K^r_q$-decomposition of $G'$, noting that if $\Phi$ is $\{0,1\}$-valued
then we can identify $\Phi$ with a (true) $K^r_q$-decomposition of $G'$. For $G \sub K^r_n$
we say that $G$ is $K^r_q$-divisible if $\pl \Phi = G$ for some  $\Phi \in \mb{Z}^{K^r_q(K^r_n)}$;
thus $K^r_q$-divisibility is a necessary condition for having a $K^r_q$-decomposition. 

\begin{theo} \label{thm:steiner} 
For all $q>r \ge 1$ there is $n_0$ so that if $n \ge n_0$ and $K^r_n$ is $K^r_q$-divisible
then $K^r_n$ has a $K^r_q$-decomposition.
\end{theo}

Our goal in this paper is to present a new proof of the existence of designs,
which is much shorter than the previous proofs, as well as being self-contained;
here one should note that  \cite{DP} is far from being self-contained, as it makes black-box use of
large parts of \cite{GKLO} and previous papers on approximate decompositions,
whereas the brevity of our paper is not at the expense of explanations,
as we give full details of the proof and (with a view to potential graduate courses)
even devote several expository sections to giving shorter and more efficient proofs of some ingredients 
that already exist in the literature (albeit with longer proofs and inferior bounds in some cases).
Our argument applies to more general decomposition problems
(such as replacing the `host' $K^r_n$ by a typical $r$-multigraph
and the `guest' $K^r_q$ by a general $r$-graph),
but for simplicity we will only present the argument for Steiner systems.

Roughly speaking, our new proof mixes some ideas from all of the three previous proofs
with a key new idea that leads to a much simpler construction of the absorber
(the key technical ingredient of all previous proofs, as explained below).
After the proof overview (following the framework of \cite{DP}) in Section \ref{sec:pf},
the new construction is presented in Sections \ref{sec:prep} and \ref{sec:abs},
and its analysis is completed in Section \ref{sec:rga}. One ingredient of the absorber
is an integral absorber, which was first developed in \cite{K}; in Section \ref{sec:int}
we present an alternative construction based on the method of \cite{KeSS}, which is not shorter
but has the advantage of being more widely applicable; indeed, our proof of Theorem \ref{thm:steiner}
does not rely on the equivalence of $K^r_q$-divisibility as defined above with the usual formulation
in terms of degree divisibility conditions (but for completeness we will include a short proof of
this equivalence, see Remark \ref{rem:div}). The remainder of the paper is essentially expository,
with some simplifications and quantitative improvements:~we discuss local decoders
and regularity boosting  in Section \ref{sec:decode+boost},
martingale concentration in Section \ref{sec:conc},
and the  clique removal process in Section \ref{sec:removal}.
The final section discusses quantitative aspects (the choice of $n_0$).

Throughout we adopt the following parameters:
\begin{equation}
\label{param}
\text{Let } q>r \ge 1, \quad k=\tbinom{q}{r}, \quad \rho = (6k)^{-2}, \quad \aA = (2q)^{-r} \rho, \quad n > n_0 = (4q)^{90q/\aA} = (4q)^{90q(2q)^r(6k)^2}.
\end{equation}

\nim{Notation.} As indicated above, we identify sets with their characteristic vectors,
and will extend this notation to multisets and `intsets' (signed multisets):
we identify $v \in \{0,1\}^X$ with the set $\{x \in X: v_x=1\}$;
we identify $v \in \mb{N}^X$ with the multiset in $X$ 
where each $x$ has multiplicity $v_x$ (for our purposes $0 \in \mb{N}$);
we identify $v \in \mb{Z}^X$ with the signed multiset in $X$ 
where each $x$ has $|v_x|$ copies of the appropriate sign.
If $G$ is a hypergraph, $v \in \mb{Z}^G$ and $e \sub V(G)$,
the neighbourhood of $e$ in $v$ is $v(e) \in \mb{Z}^{G(e)}$ 
defined by $v(e)_f = v_{e \cup f}$ for $f \in G(e)$.
Thus, if $v$ is the characteristic vector of a subgraph $G' \sub G$
then $v(e)$ is the characteristic vector of the neighbourhood $G'(e)$ of $e$ in $G'$.

We say $v \in \mb{Z}^{K^r_n}$ is $\tT$-bounded if $\sum \{ |v_e|: f \sub e \in K^r_n \} < \tT n$ 
for all $f \in \tbinom{[n]}{r-1}$. Thus, if $v$ is the characteristic vector of an $r$-graph $G \sub K^r_n$
then $v$ is $\tT$-bounded if $G$ has maximum $(r-1)$-degree at most $\tT n$.

We say that an event $E$ holds whp (with high probability) if $\mb{P}(E) > 1-e^{-n/10}$;
we adopt this non-standard terminology with an explicit constant so that one can verify
that a union bound over all whp events considered in the paper is valid
with the choice of $n_0$ in \eqref{param}. If one is not concerned with the bound on $n_0$
it suffices to note that we take a union bound over polynomially many whp events
each having failure probability $e^{-\Omega(n^c)}$ for some $c>0$.

For any $r$-graph $G$ and $p \in [0,1]$ we write $G(p)$ for the $p$-random subgraph of $G$,
where each edge of $G$ is present independently with probability $p$.

\section{Proof modulo lemmas}  \label{sec:pf}

Our strategy for proving Theorem \ref{thm:steiner} proceeds via the following five steps.
In broad outline, as in all previous proofs, it is an absorption strategy, following the framework of \cite{DP},
with simplifications to remove extra properties that are not needed for our current purpose.
The terms `absorber' and `regularity boosted'  will be explained below.

\begin{enumerate}[label=\arabic*.]
\item Randomly reserve a sparse subgraph $R \sub K^r_n$, to be used in Step 4. 
\item Find an `absorber' $A \sub K^r_n \sm R$, to be used in Step 5.
\item Find a `regularity boosted' set of $q$-cliques $H \sub K^r_q(G)$, where $G := K^r_n \sm (A \cup R)$.
\item Find edge-disjoint cliques $D \sub H$ with leave $L := G \sm \bigcup D \sub R$.
\item Find clique decompositions $D_L$ of $A \cup L$, so $D \cup D_L$ of $G$.
\end{enumerate}

Now we will state the lemmas that implement each of the above steps.
The first step uses the following lemma, which follows easily from concentration inequalities
(Chernoff bounds); we will give the proof in Section \ref{sec:rga}.

\begin{lemma} \label{lem:R} (Reserve)
For any $q,r,k,\rho,n$ as in \eqref{param},
there is some $2n^{-\rho}$-bounded $R \sub K^r_n$ such that any edge $e \in K^r_n \sm R$
belongs to at least $n^{-k\rho} n^{q-r}$ copies $Q$ of $K^r_q$ with $Q \sm \{e\} \sub R$.
\end{lemma}

The second step uses the following lemma, in which we see the meaning of `absorber'
(also called `omni-absorber' in \cite{DP}): it is a subgraph $A$ disjoint from $R$
such that $A \cup L$ has a $K^r_q$-decomposition for any $K^r_q$-divisible $L \sub R$.
The main novelty of this paper is in finding a relatively simple proof of this lemma.
Our formulation is somewhat weaker than that in \cite{DP} but sufficient for our purposes.
(They show that the boundedness of $A$ can be just a constant factor larger than that of $R$.
Their absorber is also refined, in the sense that the decomposition of $A \cup L$ only uses
cliques from some fixed family $\mc{Q}$ such that every edge is in only constantly many cliques
of $\mc{Q}$; we also obtain this property, although we do not need it, 
so for simplicity we do not include it in our statement).

\begin{lemma} \label{lem:A} (Absorber)
For any $q,r,\rho,\aA,n$ as in \eqref{param} and $n^{-\rho}$-bounded $R \sub K^r_n$, 
there is some $n^{-\aA/4}$-bounded $K^r_q$-divisible $A \sub K^r_n \sm R$ that is an absorber for $R$,
that is, $A \cup L$ has a $K^r_q$-decomposition for any $K^r_q$-divisible $L \sub R$.
\end{lemma}

The third step uses the following lemma, a special case of \cite[Lemma 6.3]{GKLO},
in which we see the meaning of `regularity boosted': given some `almost' complete $r$-graph $G$,
which in our application will be $K^r_n \sm (A \cup R)$, we find a set $H$ of $q$-cliques in $G$
(equivalently, a subgraph $H$ of the design hypergraph $K^r_q(G)$)
such that every edge $e$ of $G$ belongs to `essentially' the same number of cliques in $H$.
The point of this step is that the precise meaning of `essentially' 
involves a much better approximation than that in the precise meaning of `almost'.
Without this `boost', the error term in the first part of Step 4 below
would be too large for the second part of Step 4 to be achievable.

\begin{lemma} \label{lem:reg} (Boost)
For any $q,r,n$ as in \eqref{param},
if $G \sub K^r_n$ and $K^r_n \sm G$ is $c$-bounded with $c<2^{-3q}$ then there is 
$H \sub K^r_q(G)$ with $|H(e)| = (1/2 \pm n^{-1/3}) \tbinom{n}{q-r}$ for all $e \in V(H) = E(G)$.
\end{lemma}

The fourth step has two parts, implemented by the following two lemmas,
which we call the `nibble' and the `cover'.
Here we depart from the more general but not self-contained formulation in \cite{DP}
and use the formulation from \cite{K}. 
The term `nibble' refers to R\"odl's celebrated semi-random method,
and our statement is quite similar to his original application of this method:
we consider an $r$-graph $G$ with a very regular set $H$ of $q$-cliques
as obtained in Step 3 and find a set $D_1$ of edge-disjoint $q$-cliques
that cover most of the edges of $G$. We require the additional property
of boundedness of the $r$-graph $L_1$ of uncovered edges, 
which is well-known (see the references in \cite{K}),
but does not follow from standard proofs of the nibble, 
so for completeness we will give a short proof in Section \ref{sec:removal}.

\begin{lemma} \label{lem:nibble} (Nibble)
Fix $q,r,k,\rho,n$ as in \eqref{param}.
Suppose $G \sub K^r_n$ and $H \sub K^r_q(G)$ with $|G| > \tfrac12 \tbinom{n}{r}$
and $|H(e)| =  (1 \pm n^{-1/3}) \tfrac12 \tbinom{n}{q-r}$ for all $e \in V(H) = E(G)$.
Then there is a set $D_1 \sub H$ of edge-disjoint $q$-cliques
such that $L_1 :=  G \sm \bigcup D_1$ is $n^{-3k\rho}$-bounded.
\end{lemma} 

Next we state the `cover' lemma, which shows that $r$-graph $L_1$ of uncovered edges
from the nibble step can be covered by a set $D_2$ of  edge-disjoint $q$-cliques,
each of which covers exactly one edge of $L_1$ and has all remaining edges 
in the $r$-graph $R$, which was reserved in Step 1 precisely to provide
many suitable cliques for this step. This is accomplished by a general purpose random greedy algorithm
presented in Section \ref{sec:rga} which will be used in various contexts throughout the paper.
 
\begin{lemma} \label{lem:cover} (Cover)
Fix $q,r,k,\rho,n$ as in \eqref{param}.
Let $R \sub K^r_n$ be $n^{-\rho}$-bounded and such that any edge $e \in K^r_n \sm R$
belongs to at least $n^{-k\rho} n^{q-r}$ copies $Q$ of $K^r_q$ with $Q \sm \{e\} \sub R$.
Suppose $L_1 \sub K^r_n \sm R$ is $n^{-3k\rho}$-bounded.
Then there is a set $D_2 = \{Q_e: e \in L_1\} \sub K^r_q(G \sm A)$ 
of edge-disjoint $q$-cliques with each $e \in Q_e$ and $Q_e \sm \{e\} \sub R$.
\end{lemma} 

We conclude this section by proving the main theorem assuming the above lemmas.

\begin{proof}[Proof of Theorem \ref{thm:steiner}]
Let $q,r,k,\rho,\aA,n$ be as in \eqref{param} where $K^r_n$ is $K^r_q$-divisible.

Apply Lemma \ref{lem:R} to obtain some $n^{-\rho}$-bounded `reserve' $R \sub K^r_n$ such that any edge $e \in K^r_n \sm R$
belongs to at least $n^{-k\rho} n^{q-r}$ copies $Q$ of $K^r_q$ with $Q \sm \{e\} \sub R$.

Apply Lemma \ref{lem:A} to find some $n^{-\aA/4}$-bounded $K^r_q$-divisible `absorber' $A \sub K^r_n \sm R$ 
such that $A \cup L$ has a $K^r_q$-decomposition for any $K^r_q$-divisible $L \sub R$.

Apply Lemma \ref{lem:reg} with $G = K^r_n \sm (A \cup R)$, using $n^{-\aA/4} + n^{-\rho} < 2^{-3q}$, to obtain
$H \sub K^r_q(G)$ with $|H(e)| = (1 \pm n^{-1/3}) \tfrac12 \tbinom{n}{q-r}$ for all $e \in V(H) = E(G)$.

Apply Lemma \ref{lem:nibble} to find a set $D_1 \sub H$ of edge-disjoint $q$-cliques
such that $L_1 :=  G \sm \bigcup D_1$ is $n^{-3k\rho}$-bounded.

Apply Lemma \ref{lem:cover} to find a set $D_2 = \{Q_e: e \in L_1\} \sub K^r_q(G \sm A)$ 
of edge-disjoint $q$-cliques with each $e \in Q_e$ and $Q_e \sm \{e\} \sub R$.

Let $L = (K^r_n \sm A) \sm \bigcup D$, where $D = D_1 \cup D_2$.
Then $L \sub R$ is $K^r_q$-divisible, 
as $K^r_q$-divisibility is closed under integer linear combinations 
and holds for $K^r_n$, $ \bigcup D$ and $A$.

Now $A \cup L$ has a $K^r_q$-decomposition $D_L$ by choice of $A$,
so $D \cup D_L$ is a $K^r_q$-decomposition of $K^r_n$. 
\end{proof}

\section{Preparing to absorb} \label{sec:prep}

Having reduced the theorem to the lemmas,
the main remaining hurdle for the proof will be proving
the absorber lemma required for Step 2 above.
The construction of the absorber will be presented in the next section,
following some preparations in this section. 
In the first subsection we introduce some notation and motivate the construction
by giving an informal description of the objects that we will require.
The second subsection introduces the clique exchange tool
by stating a lemma on its existence and properties
and describing its later applications in the paper.
We prove this lemma in the final subsection.

\subsection{Motivation} \label{sub:motiv}

Given a set $\mc{Q} \sub K^r_q(K^r_n)$ of $q$-cliques in $K^r_n$,
we write $\pl \mc{Q}$ for the $r$-multigraph $M \in \mb{N}^{K^r_n}$ 
where each $e \in K^r_n$ has multiplicity $M_e = |\{Q \in \mc{Q}: e \in Q\}|$.

More generally, for $\Phi \in \mb{Z}^{K^r_q(K^r_n)}$ we define $\pl \Phi \in \mb{Z}^{K^r_n}$
by $(\pl \Phi)_e = \sum \{ \Phi_Q: e \in Q \in K^r_q(K^r_n) \}$.
We think of such $\Phi$ as a characteristic vector of an integral decomposition,
containing $\Phi_Q$ signed copies of each clique $Q$.

The first step in the construction in subsection \ref{sub:construct}
will be to find $\mc{Q}_1 \sub K^r_q(K^r_n)$ 
such that $\pl \mc{Q}_1$ is suitably bounded and $\bigcup \mc{Q}_1$ 
is an `integral absorber' for $R$, meaning that for any
$K^r_q$-divisible $J \in \mb{Z}^{K^r_n}$ supported in $R$
there is $\Phi \in \mb{Z}^{\mc{Q}_1}$ with $\pl \Phi = J$.
In particular, we have such $\Phi$ for $K^r_q$-divisible $J = L \sub R$.

We will then use the clique exchanges defined below
to augment $\mc{Q}_1$ to a larger collection $\mc{Q} = \mc{Q}^+ \cup \mc{Q}^-$,
which can be used to convert the integral decomposition $\Phi$ of $L$
into a signed decomposition $L = \pl D^+ - \pl D^-$ with $D^\pm \sub \mc{Q}^\pm$.
Furthermore, $\mc{Q}^\pm$ will be suitably bounded
and the cliques in $\mc{Q}^-$ will be edge-disjoint from $R$ and each other.

The absorber will be defined as $A = \bigcup \mc{Q}^-$. 
Then for any $K^r_q$-divisible $L \sub R$, to decompose $A \cup L$ 
we will solve $L = \pl D^+ - \pl D^-$ as indicated above
and note that $D^+ \cup \mc{Q}^- \sm D^-$ decomposes $A \cup L$.

\subsection{Clique exchange} \label{sub:ex}

Our basic building block for various clique exchange operations throughout the paper
consists of an $r$-graph $\OO$ with two $K^r_q$ decompositions $\Ups^+$ and $\Ups^-$.
We have a designated clique $\hQ^+ \in \Ups^+$ such that the cliques $\hQ^e \in \Ups^-$
that share an edge with $\hQ^+$ are `maximally disjoint', in the sense of the following lemma.

\begin{lemma} \label{lem:OO}
There is an $r$-graph $\OO$ with $|\OO| \le 3(2q)^r \tbinom{q}{r}^2$ 
and two $K^r_q$ decompositions $\Ups^+$ and $\Ups^-$ of $\OO$ such that 
for some  $\hQ^+ \in \Ups^+$ and $\{ \hQ^e : e \in \hQ^+ \} \sub \Ups^-$,
writing $F := V(\hQ^+) \cup \bigcup_e V(\hQ^e)$, we have 
\begin{enumerate}
\item each $\hQ^e \cap \hQ^+ = \{e\}$ and
$\{ V(\hQ^e) \sm V(\hQ^+) : e \in \hQ^+\}$ are pairwise disjoint,
\item for any edge $e' \in \OO$ we have $e' \cap F \sub V(\hQ^+)$
or $e' \cap F \sub V(\hQ^e)$ for some $e \in \hQ^+$.
\end{enumerate}
\end{lemma}

The $r$-graph $\OO$ in Lemma \ref{lem:OO} will have three different applications in this paper.
Some applications will only use some of the above properties of $\OO$, but for simplicity of exposition
we sacrifice some efficiency by using the same construction $\OO$ for all applications.
The first application will be in Step 3 (Splitting) of the construction in subsection \ref{sub:construct}.
Here we will modify some integral decomposition $\Phi \in \mb{Z}^{K^r_q(K^r_n)}$
without changing the value of $\pl \Phi$ so as to eliminate a copy of some clique $Q$.
To do so, we fix any copy $\phi(\OO)$ of $\OO$ with $\phi(\hQ^+)=Q$
and define $\Phi' \in \mb{Z}^{K^r_q(K^r_n)}$ 
by $\Phi' = \Phi + \phi(\Ups^-) - \phi(\Ups^+)$,
meaning that each $\Phi'_{Q'}$ is $\Phi_{Q'} + 1$ for $Q' \in \phi(\Ups^-)$,
or is $\Phi_{Q'} - 1$ for $Q' \in \phi(\Ups^+)$, or is  $\Phi_{Q'}$ otherwise.
Then $\pl \Phi' = \pl \Phi$, as $\phi(\Ups^+)$ and $\phi(\Ups^-)$
cover the same set of edges, so their contributions cancel.
Furthermore, as $\hQ^+ \in \Ups^+$ we have $\Phi'_Q = \Phi_Q - 1$;
if $\Phi_Q>0$ we interpret this as removing a copy of $Q$ from $\Phi$ 
and replacing it by an equivalent set of signed cliques. 
Similarly, if $\Phi_Q<0$ then we consider $\Phi' = \Phi - \phi(\Ups^-) + \phi(\Ups^+)$,
which we interpret as removing a negative copy of $Q$ from $\Phi$
and replacing it by an equivalent set of signed cliques. 
For this application, the only important property in Lemma \ref{lem:OO}
is the first part of (i), which implies that 
\begin{enumerate}
\item the cliques replacing $Q$
in the above operation all intersect $Q$ in at most $r$ vertices.
\end{enumerate}

For the second application, we rename $\hQ^{e_0}$ as $\hQ^-$
and will not use the properties of the other $\hQ^e$ for $e \ne e_0$;
these will only be needed once, in our third application, the proof of Lemma  \ref{lem:extcol}.4.
Thus we focus on two cliques $\hQ^\pm \in \Ups^\pm$ with $\hQ^+ \cap \hQ^- = \{e_0\}$,
such that every edge of $\OO$ contained in $V(\hQ^+) \cup V(\hQ^-)$ is contained in $\hQ^+$ or $\hQ^-$
(the latter property is needed for `admissibility' in the sense of Definition \ref{def:process} below).
We use $\hQ^\pm$ in Step 4 (Elimination) of the construction in subsection \ref{sub:construct}
to modify some integral decomposition $\Phi \in \mb{Z}^{K^r_q(K^r_n)}$ without changing 
the value of $\pl \Phi$ so as to eliminate a `cancelling pair' $Q^+ - Q^-$,
consisting of two cliques of opposite sign in $\Phi$ that intersect exactly in some edge $e$.
To do so, we fix any copy $\phi(\OO)$ of $\OO$ with $\phi(\hQ^\pm)=Q^\pm$
and replace $\Phi$ by  $\Phi' = \Phi + \phi(\Ups^-) - \phi(\Ups^+)$.
Then $\pl \Phi' = \pl \Phi$ is unchanged, $\Phi'_{Q^+} = \Phi_{Q^+} - 1$ and $\Phi'_{Q^-} = \Phi_{Q^-} + 1$.
If $\Phi_{Q^+}>0$ and $\Phi_{Q^-}<0$ then the interpretation is that we have removed
a positive copy of $Q^+$ and a negative copy of $Q^-$ and replaced them by an equivalent set of signed cliques. 
Two crucial properties of this construction are that
\begin{enumerate}
\item the new signed cliques do not use the common edge $e$ of $Q^+$ and $Q^-$, and
\item the new negative cliques intersect $Q^-$ in at most one edge
and are edge-disjoint from each other and $Q^+$.
\end{enumerate}

\subsection{Construction}

We conclude the section with the construction of $\OO$.

\begin{proof}[Proof of Lemma \ref{lem:OO}]
We will construct $\OO$ by gluing together several copies of $\OO_0 := K^r_q(p)$,
which is our notation for the $p$-blowup of $K^r_q$, obtained from $K^r_q$ 
by replacing each vertex $i \in [q]$ by a set $X_i$ of $p$ vertices,
and each edge $e$ by all $p^r$ edges obtained by replacing 
each vertex $i$ of $e$ by some $x \in X_i$.
We choose $p$ to be a prime in $[q,2q]$, which exists by Bertrand's postulate,
and identify each $X_v$ with the finite field $\mb{F}_p$.
We note that $|\OO_0| = p^r \tbinom{q}{r} \le (2q)^r \tbinom{q}{r}$.

Next we describe two $K^r_q$-decompositions of $\OO_0$, denoted $\Ups^\pm_0$.
We will identify each $q$-clique $Q$ with a vector $v \in \mb{F}_p^q$,
where $v_i \in X_i = \mb{F}_p$ is the vertex of $Q$ in $X_i$.
We let $M \in \mb{F}_p^{q \times r}$ be a Vandermonde matrix,
where row $i$ of $M$ is $(1,y_i,y_i^2,\dots,y_i^{r-1})$
for some $y_i \in \mb{F}_p$, and the rows are distinct
(which is possible as $p \ge q$). We let $\Ups^+_0$ consist
of all $q$-cliques corresponding to vectors 
$v=Mu \in \mb{F}_p^q$ for some $u \in \mb{F}_p^q$.
We let $\Ups^-_0$ consist of all $q$-cliques corresponding to vectors 
$v=v^0+Mu \in \mb{F}_p^q$ for some $u \in \mb{F}_p^q$,
where $v^0_i = 0$ for $i \in [r]$ and $v^0_i=1$ for $i \in [q] \sm [r]$.

To see that $\Ups^+_0$ is a clique decomposition of $\OO_0$,
we need to show that for each edge $e \in \OO_0$
there is a unique clique in $\Ups^+_0$ that contains $e$.
Let $I \in \tbinom{[q]}{r}$ index the parts of $\OO_0$
that $e$ intersects and write $e = \{x_i: i \in I\}$.
Consider any clique $v=Mu$ in $\Ups^+_0$ containing $e$,
meaning that $v_i = x_i$ for all $i \in I$.
Let $M_I$ be the square submatrix of $M$ corresponding to the rows of $I$.
Then $M_I$ is non-singular, as up to sign $\det M_I$ is the product
of all differences $y_i-y_{i'}$ for $\{i,i'\} \sub I$,
so $u$ is uniquely determined by $u=M_I^{-1}e$, 
regarding $e$ as a vector in $\mb{F}_p^I$.
Similarly, $\Ups^-_0$ is a clique decomposition of $\OO_0$,
as for $e$ as above the clique $v=v^0+Mu$ containing $e$
is uniquely determined by $u=M_I^{-1}(e-v^0_I)$,
where $v^0_I \in \mb{F}_p^I$ is the restriction of $v^0$ to coordinates $I$.

We let $\hQ^\pm_0 \in \Ups^\pm_0$ denote the cliques containing
the edge $e_0$ that contains element $0$ in parts $X_1,\dots,X_r$,
so $\hQ^+_0$ corresponds to the zero vector in $ \mb{F}_p^q$
and $\hQ^-_0$ corresponds to $v^0$. We note that $\hQ^+_0 \cap \hQ^-_0 = \{e_0\}$.

Next we define a gluing procedure for combining two $r$-graphs $\OO_1,\OO_2$
each having two  $K^r_q$-decompositions $\Ups^\pm_i$
and designated $q$-cliques $\hQ^-_1 \in \Ups^-_1$,  $\hQ^+_2 \in \Ups^+_2$.
A glued $r$-graph $\OO_3$ is obtained from disjoint copies of $\OO_1$ and $\OO_2$ 
by identifying $\hQ^-_1$ and $\hQ^+_2$ according to some specified bijection.
We obtain two  $K^r_q$-decompositions $\Ups^\pm$ of $\OO_3$ 
by setting  $\Ups^+ = \Ups^+_1 \cup \Ups^+_2  \sm \{\hQ^+_2\}$
and $\Ups^- = \Ups^-_2 \cup \Ups^-_1  \sm \{\hQ^-_1\}$.
To see that  $\Ups^+$ decomposes $\OO_3$, note that  $\Ups^+_1$ decomposes $\OO_1$ 
and $\Ups^+_2  \sm \{\hQ^+_2\}$ decomposes $\OO_3 \sm \OO_1 = \OO_2 \sm \hQ^+_2$;
similarly $\Ups^-$ decomposes $\OO_3$.  As signed decompositions,
we have $\Ups^+ - \Ups^- = (\Ups^+_1 - \Ups^-_1) + (\Ups^+_2 - \Ups^-_2)$,
where $\hQ^-_1 = \hQ^+_2$ appears in $\Ups^-_1$ and $\Ups^+_2$, so not in $\Ups^\pm$.

We define the $r$-graph $\OO$ and decompositions $\Ups^\pm$ required for the lemma
by gluing $2\tbinom{q}{r}$ copies of $\OO_0$ as follows. We start with $\OO' = \OO^0_0$ being
a copy of $\OO_0$ and proceed in $\tbinom{q}{r}$ rounds, in each of which we update $\OO'$ by gluing
two copies of $\OO_0$, then we let $\OO$ be the final $r$-graph $\OO'$. 
The rounds correspond to edges $e$ 
of the designated clique $\hQ^{0+}_0 \in \Ups^{0+}_0$.
Given the current $r$-graph $\OO'$, with some $K^r_q$-decompositions $\Ups'{}^\pm$,
and some edge $e$, we let $\hQ'{}^-(e)$ denote the unique clique in $\Ups'{}^-$ containing $e$.
We define the $e$-gluing operation on $\OO'$ by gluing on a copy of $\OO_0$ according to some
bijection of $\hQ'{}^-(e)$ with $\hQ^+_0$ identifying $e \in  \hQ'{}^-(e)$ with $e_0 \in \hQ^+_0$.
In each round $e$ we perform two $e$-gluing operations. 
This completes the definition of $\OO$ and $\Ups^\pm$;
it remains to check that they have the required properties.

Fix some $e \in \hQ^{0+}_0$ and let $(\OO_i,\Ups^\pm_i)$ for $i=1,2,3$ denote $(\OO',\Ups'{}^\pm)$ 
at the start of round $e$, then after one $e$-gluing, then after the second $e$-gluing.
We let $\hQ^\pm_i(e)$ denote the clique of $\Ups^\pm_i$ containing $e$.
Each $e$-gluing removes the current clique $\hQ'{}^-(e)$ of $\Ups'{}^-$ containing $e$ 
and replaces it by the clique  $\hQ^-_0$ in the glued copy of $\Ups^-_0$ containing $e_0$, 
which by construction of $\Ups^\pm_0$ intersects $\hQ^+_0 = \hQ'{}^-(e)$ precisely in $e_0=e$.
Thus $\OO_2$ is obtained by identifying $\OO_1$ and a copy $\OO^1_0$ of $\OO_0$
so that $\hQ^-_1(e)=\hQ^{1+}_0$ and $\hQ^-_2(e) = \hQ^{1-}_0$,
so $V(\OO^1_0) \cap V(\OO_1) = V(\hQ^-_1(e))$
and $V(\hQ^-_2(e)) \cap V(\OO_1) = e$.
Then  $\OO_3$ is obtained by identifying $\OO_2$ and a copy $\OO^2_0$ of $\OO_0$
so that $\hQ^-_2(e)=\hQ^{2+}_0$ and $\hQ^-_3(e) = \hQ^{2-}_0$,
so $V(\OO^2_0) \cap V(\OO_2) = V(\hQ^-_2(e))$
and $V(\hQ^-_3(e)) \cap V(\OO_2) = e$,
giving $V(\OO^2_0) \cap V(\OO_1) = V(\OO^2_0) \cap V(\hQ^-_2(e)) \cap V(\OO_1) = e$.
In words, after two $e$-gluings, the new clique $\hQ'{}^-(e)=\hQ^-_3(e)$ of $\Ups'{}^-$ containing $e$ at the end of round $e$
belongs to a copy $\OO^2_0$ of $\OO_0$ that intersects the old $\OO'$ from the start of round $e$
precisely in $e$ and is otherwise vertex-disjoint. This property will be preserved by all subsequent modifications
of $\OO'$, as these do not change $\hQ'{}^-(e)$ and all copies of $\OO_0$ to be glued later
will be vertex-disjoint from $V(\OO^2_0) \sm e$.
 
The designated cliques in $\Ups^\pm$ required in the statement of the lemma 
are $\hQ^+ := \hQ^{0+}_0 \in \Ups^+$ 
(which was present in $\Ups^{0+}_0$ and has never been replaced)
and for each $e \in \hQ^+$ the clique $\hQ^e := \hQ^-(e) \in \Ups^-$ containing $e$.
Then each $\hQ^e \cap \hQ^+ = \{e\}$ and
$\{ V(\hQ^e) \sm V(\hQ^+) : e \in \hQ^+\}$ are pairwise disjoint.
Also, writing  $F := V(\hQ) \cup \bigcup_e V(\hQ^e)$, 
for any edge $e' \in \OO$ we have $e' \cap F \sub V(\hQ^+)$
or $e' \cap F \sub V(\hQ^e)$ for some $e \in \hQ^+$.
Indeed, if $e' \cap F \not\sub V(\hQ^+)$ then $e' \cap (V(\hQ^e) \sm e) \ne \es$
for some $e \in V(\hQ^+)$, which implies that $e'$ is in the second copy $\OO^2_0$
of $\OO_0$ in round $e$, and so $e' \cap F \sub V(\OO^2_0) \cap F = V(\hQ^e)$, as required.
\end{proof}

\section{The absorber} \label{sec:abs}

In the first subsection of this section we describe the construction of the absorber,
assuming the existence of the integral absorber $\mc{Q}_1$, which will be proved in Lemma \ref{lem:Aint},
and that certain random greedy algorithms are successful with high probability (whp), 
which will be proved in Section \ref{sec:rga}.
In the second subsection, given these assumptions,
we prove the absorber lemma (Lemma \ref{lem:A}).
This will complete the proof of the main theorem modulo previously known results
(which we will prove later for the sake of completeness).

\subsection{The construction} \label{sub:construct}

Let $q,r,\rho,\aA,n$ be as in \eqref{param}
and suppose $R \sub K^r_n$ is $n^{-\rho}$-bounded.

1. (Integral absorber)
Use Lemma \ref{lem:Aint}
to find $\mc{Q}_1 \sub K^r_q(K^r_n)$ such that 
$\pl \mc{Q}_1$ is $n^{-\aA/2}$-bounded, 
every edge of $K^r_n$ is in at most two cliques of $\mc{Q}_1$,
and for any $K^r_q$-divisible $J \in \mb{Z}^{K^r_n}$ supported in $R$
there is $\Phi \in \mb{Z}^{\mc{Q}_1}$ with $\pl \Phi = J$.

2. (Local decoders)
Apply a random greedy algorithm to choose $(q+r)$-sets $Z_e$ containing $e$ 
for each $e \in \bigcup \mc{Q}_1$, so that each $r$-graph $\tbinom{Z_e}{r} \sm \{e\}$
is disjoint from all others and from $\bigcup \mc{Q}_1$.
Let $\mc{Q}_2 = \bigcup_{e \in \bigcup \mc{Q}_1} K^r_q(\tbinom{Z_e}{r})$.

3. (Splitting)
Let $Q_1,\dots,Q_t$ be a sequence in $K^r_q(K^r_n)$ consisting of
$2^{q+1} r!$ copies of each $q$-clique in $\mc{Q}_1 \cup \mc{Q}_2$,
where $2^q r!$ are labelled $+$ and $2^q r!$ are labelled $-$.
Apply a random greedy algorithm to choose copies $\OO_i = \phi_i(\OO)$ of $\OO$ 
with $\phi_i(\hQ^+) = Q_i$, where each $\OO_i \sm Q_i$ is edge-disjoint from all others 
and from $A_0 :=  \bigcup(\mc{Q}_1 \cup \mc{Q}_2)$,
and $V(\OO_i) \sm V(Q_i)$ is disjoint from $V(\OO_j) \sm V(Q_j)$
whenever $Q_i$ and $Q_j$ share an edge.

4. (Elimination) 
For each $i \in [t]$ and $Q' \in \Ups^- \cup \Ups^+ \sm \{\hQ^+\}$ we call $\phi_i(Q')$ a splitting clique,
with the same sign as $Q_i$ if $Q' \in \Ups^-$ or the opposite sign if $Q' \in \Ups^+$;
if $Q'$ shares an edge with $\hQ^+$ we call $\phi_i(Q')$ near, otherwise it is far.
Let $(Q^-_1,Q^+_1),\dots, (Q^-_{t'},Q^+_{t'})$ be a sequence consisting of all pairs
of oppositely signed near cliques with a common edge.
Apply a random greedy algorithm to choose copies $\OO'_i=\phi'_i(\OO)$
with $\phi'_i(\hQ^\pm)=Q^\pm_i$ where each $\OO'_i \sm (Q^-_i \cup Q^+_i)$ 
is edge-disjoint from all others and from all previously chosen cliques.

5. (Further Elimination)
For each $i \in [t']$ and $Q' \in \Ups^\pm \sm \{\hQ^-,\hQ^+\}$
we call $\phi'_i(Q')$ an elimination clique with the opposite sign to $Q'$.
We call a negative elimination clique bad (for $e$) if it shares some edge $e$ with some negative near clique.
Let $(Q'{}^-_{\!\! 1},Q'{}^+_{\!\! 1}),\dots, (Q'{}^-_{\!\! t''},Q'{}^+_{\!\! t''})$ be a sequence consisting of all pairs
where each $Q'{}^-_{\!\! i}$ is a negative elimination clique that is bad for some $e_i$
and $Q'{}^+_{\!\! i}$ is the positive splitting clique containing $e_i$ (which is unique and far).
Apply a random greedy algorithm to choose copies $\OO''_i=\phi''_i(\OO)$
with $\phi''_i(\hQ^\pm)=Q'{}^\pm_{\!\! i}$ where each $\OO''_i \sm (Q'{}^-_{\!\! i} \cup Q'{}^+_{\!\! i})$ 
is edge-disjoint from all others and from all previously chosen cliques.

6. (Conclusion)
For each $Q'{}^-_{\!\! i}$ as in Step 5 and each $Q' \in \Ups^\pm \sm \{ \hQ^-, \hQ^+\}$ 
we call $\phi''_i(Q')$ a further elimination clique of the opposite sign to $Q'$.
We define $\mc{Q} = \mc{Q}^+ \cup \mc{Q}^-$,
where $\mc{Q}^+$ contains all positive splitting cliques,
positive elimination cliques and positive further elimination cliques,
and $\mc{Q}^-$ contains all negative far splitting cliques, 
negative good elimination cliques and negative further elimination cliques.
We define $A := \bigcup \mc{Q}^-$.

We note the following properties of the construction. 

\begin{enumerate}
\item
The cliques in $\mc{Q}_1$ or $\mc{Q}_2$ and the negative near cliques 
are not used in the final set $\mc{Q}$.
\item
After Step 3, the negative far splitting cliques are all edge-disjoint,
and for any edge $e$ covered say $x$ times by $\mc{Q}_1 \cup \mc{Q}_2$ 
we have $2^q r! x$ negative near cliques that intersect pairwise precisely in $e$ 
and do not share any edge with any other negative splitting cliques.
This uses the important property of splitting  noted in subsection \ref{sub:ex},
which also shows that for any edge $e \notin A_0 := \bigcup (\mc{Q}_1 \cup \mc{Q}_2)$
covered by a negative near clique, all positive splitting cliques containing $e$ are far.
\item
After Step 4, the good negative elimination cliques and negative far splitting cliques 
are all edge-disjoint from each other and from $A_0 := \bigcup (\mc{Q}_1 \cup \mc{Q}_2)$, 
by the two crucial properties of elimination noted in subsection \ref{sub:ex}.
Indeed, the first property gives disjointness from $A_0$, and the second property shows that
we have defined good/bad correctly, in that the only source of edge multiplicity in negative
elimination cliques comes from multiply used edges in some negative near clique. 
Furthermore, for each edge $e \notin A_0$ covered by a negative near clique for some $e' \in A_0$
that is in say $x$ positive near cliques we have $x$ bad negative elimination cliques
that intersect precisely in $e$ and do not share any edge with any other negative
elimination cliques or any negative far splitting cliques. 
\item 
In Step 5, when we consider some negative elimination clique $Q'{}^-_{\!\! i}$ that is bad for some $e_i$,
we have $e_i \notin A_0$ by (iii) so by (ii) the positive splitting clique $Q'{}^+_{\!\! i}$ 
containing $e_i$ is indeed unique and far. Also, as in (iii), the two crucial properties 
of elimination show that the negative further elimination cliques are edge-disjoint,
as they do not use the bad $e_i$ and the only previous cliques with which they share edges
are the bad negative elimination cliques $Q'{}^-_{\!\! i}$, which by (iii) become edge-disjoint
if we exclude bad edges (their intersections with the negative near cliques),
and each other edge in  $Q'{}^-_{\!\! i}$ is covered by a unique negative further elimination clique.
The algorithm also chooses the negative further elimination cliques to not share any edges with
previous cliques apart from the above intersections with bad negative elimination cliques,
so they are edge-disjoint from the negative far splitting cliques and the negative good elimination cliques.
Thus $A := \bigcup \mc{Q}^-$ as in Step 6 is a union of edge-disjoint cliques,
and is disjoint from $R \sub \bigcup \mc{Q}_1$.
\end{enumerate}

\subsection{Using the absorber}

In this subsection we prove Lemma \ref{lem:A},
assuming that the construction described in subsection \ref{sub:construct} is possible, 
as will be proved later in the paper.
We require the following `local decodability' lemma.


\begin{lemma} \label{lem:decode}
Fix $e \in K^r_{r+q}$. There is $\Psi \in \mb{Z}^{K^r_q(K^r_{r+q})}$ such that
$\pl \Psi = N \cdot 1_e$ with $N:=r!k$ and each $|\Psi_Q| \le 2^q r!$.
\end{lemma}

For completeness, we include a proof of Lemma \ref{lem:decode}
in Section \ref{sec:decode+boost}, following Wilson \cite{W}.

\begin{proof}[Proof of Lemma \ref{lem:A}] 
Let $q,r,\rho,\aA,n$ be as in \eqref{param}
and suppose $R \sub K^r_n$ is $n^{-\rho}$-bounded.
Obtain $A = \bigcup \mc{Q}^-$ from the procedure in the previous subsection.
We will show in Corollary \ref{cor:A} below that whp it is successful
and produces $A$ that is $n^{-\aA/4}$-bounded. 
Also, $A$ is $K^r_q$-divisible, as it is an edge-disjoint union of cliques.
It remains to show that it is an absorber:
given some $K^r_q$-divisible $L \sub R$,
we need to find a $K^r_q$-decomposition of $A \cup L$. 

By Step 1 there is $\Phi \in \mb{Z}^{\mc{Q}_1}$ with $\pl \Phi = L$.
Let $\Phi^1 \in [-N/2,N/2]^{\mc{Q}_1}$ with $N \mid \Phi_Q - \Phi^1_Q$ for all $Q \in \mc{Q}_1$.
Then $J := \pl \Phi - \pl \Phi^1 = L - \pl \Phi^1 \in N\mb{Z}^{K^r_n}$,
and $J$ is supported in $R \cup \bigcup \mc{Q}_1 = \bigcup \mc{Q}_1$.
As each edge $e$ is in at most two cliques of $\mc{Q}_1$ 
we have $|\pl \Phi^1_e| \le N$, so $J_e \in \{-N,0,N\}$.
For each $e$ let $\Psi^e$ be a copy of $\Psi$ as in Lemma \ref{lem:decode}
on $Z_e$ from Step 2 with $\pl \Psi = N \cdot 1_e$.
Then $J = \pl \Phi^2$, where $\Phi^2 \in \mb{Z}^{\mc{Q}_2}$ is defined by
$\Phi^2 = \sum_{e \in \bigcup \mc{Q}_1} N^{-1} J_e \Psi^e$.

Now $L = \pl \Phi$, where $\Phi = \Phi^1 + \Phi^2 \in \mb{Z}^{\mc{Q}_1 \cup \mc{Q}_2}$
and $|\Phi_Q| \le \max\{N, 2^q r!\} = 2^q r!$ for all $Q \in  \mc{Q}_1 \cup \mc{Q}_2$.

We think of $\Phi$ as a subsequence of the signed cliques $Q_1,\dots,Q_t$ in Step 3.
Whenever $\Phi$ includes some $Q_i = \phi_i(\hQ)$ labelled $\pm$ 
then we add $\pm(\phi_i(\Ups^-)-\phi_i(\Ups^+))$, thus obtaining $\Phi'$ with $\pl \Phi' = \pl \Phi = L$
such that $\Phi'$ is supported on the splitting cliques and is $\{-1,0,1\}$-valued.
We think of the nonzero entries of $\Phi'$ as a set of positive cliques and a set of negative cliques,
where any $e \in L$ is in one more positive clique than negative clique of  $\Phi'$
and any $e \notin L$ is in the same number of positive cliques as negative cliques of $\Phi'$.

Any edge $e \in A_0$ is only in near splitting cliques of $\Phi'$, which we arbitrarily decompose
into `cancelling pairs' of oppositely signed cliques, plus one leftover positive clique if $e \in L$.
Each cancelling pair is of the form $Q^\pm _i$ for some $i \in [t']$ as in Step 4,
for which we chose some $\OO'_i=\phi'_i(\OO)$ with $\phi'_i(\hQ^\pm)=Q^\pm_i$.
We obtain $\Phi''$ from $\Phi'$ by adding $\phi'_i(\Ups^-) - \phi'_i(\Ups^+)$ for each such $i$,
so that $\pl \Phi'' = \pl \Phi' = L$ and all cancelling pairs have been eliminated,
being replaced by some of the elimination cliques $\phi'_i(Q')$ 
for $Q' \in \Ups^\pm \sm \{\hQ^-,\hQ^+\}$ considered in Step 5.

The only remaining splitting cliques in $\Phi''$ are positive, so in $\mc{Q}$.
Any good negative elimination cliques in $\Phi''$ are also in $\mc{Q}$.
Any bad negative elimination cliques in $\Phi''$ are of the form $Q'{}^-_{\!\! i}$,
for which we define $e_i$ and $Q'{}^+_{\!\! i}$ as in Step 5,
recalling that $Q'{}^+_{\!\! i}$ is far and is the unique positive splitting clique containing $e_i$.
Furthermore, no positive elimination clique $Q'$ can contain $e_i$, as it was chosen to be 
edge-disjoint from all previous cliques except for possibly sharing an edge with some $Q^+_j$
from Step 4, which is a positive near clique, so edge-disjoint from all positive far cliques.
As $\pl \Phi''_{e_i} > -1$, there must be a positive clique in $\Phi''$ containing $e_i$,
which can only be $Q'{}^+_{\!\! i}$.

We obtain $\Phi'''$ from $\Phi''$ by adding $\phi''_i(\Ups^-) - \phi''_i(\Ups^+)$ for each such $i$,
so that $\pl \Phi''' = \pl \Phi'' = L$ and all bad negative elimination cliques have been eliminated,
being replaced by some of the further elimination cliques $\phi''_i(Q')$ 
for $Q' \in \Ups^\pm \sm \{\hQ^-,\hQ^+\}$ considered in Step 6, which are in $\mc{Q}$.
Thus we have $L = \pl \Phi'' = \pl D^+ - \pl D^-$ with $D^\pm \sub \mc{Q}^\pm$.
Recalling that $\mc{Q}^-$ decomposes $A$, we conclude that
$D^+ \cup \mc{Q}^- \sm D^-$ decomposes $A \cup L$.
\end{proof}

\section{Typicality and random greedy algorithms}  \label{sec:rga}

Throughout the paper we require the following inequalities on concentration of probability, 
which we will refer to as Chernoff bounds, as they are slight generalisations of the usual 
Chernoff bounds for large deviations of binomial random variables; 
for completeness we include a proof in Section \ref{sec:conc}.

\begin{lemma}\label{lem:pseudobin}
Let $X = \sum_{i=1}^n X_i$ where $X_i$ are random variables with each $|X_i| \le C$.
\begin{enumerate}[label=\arabic*.]
\item If $\{X_i\}$ are independent and $\mb{E}X = \mu$
then $\mb{P}(|X-\mu|>c\mu) \le 2e^{-\mu c^2/2(1+2c)C}$.
\item If $\sum_i \mb{E}[|X_i| \mid \{X_j: j<i\}] \le \mu$ then
$\mb{P}(|X|>(1+c)\mu) \le 2e^{-\mu c^2/2(1+2c)C}$.
\end{enumerate}
\end{lemma}

Next we define typicality, which expresses a pseudorandomness property of $r$-graphs,
namely that the sizes of common neighbourhoods are roughly what one would expect
if edges were chosen randomly. The following definition is essentially that in \cite{K};
the accompanying lemma shows that (unsurprisingly) random $r$-graphs are typical.

\begin{defn} \label{def:typ}
Suppose $G$ is an $r$-graph on $[n$]. The density of $G$ is $d(G) = |G| \tbinom{n}{r}^{-1}$.
We say that $G$ is $(c,h)$-typical if any $A \sub \tbinom{[n]}{r-1}$ with $|A| \le h$
has common neighbourhood of size close to its expectation in a $d(G)$-random $r$-graph,
namely $\bsize{\bigcap_{S \in A} G(S)} = (1 \pm c) d(G)^{|A|}n$.
\end{defn}

\begin{lemma} \label{lem:randomtyp}
Suppose $r,s,n \in \mb{N}$ with $n>2^{9rs}$ and
$G \sim K^r_n(p)$ is $p$-random in $K^r_n$ with $p>n^{-1/2s}$.
Then whp $G$ is $(n^{-0.1},s)$-typical.
\end{lemma}

\begin{proof}
Fix $A \sub \tbinom{[n]}{r-1}$ with $|A| \le s$. Any $v \notin \bigcup A$ 
belongs to $\bigcap_{S \in A} G(S)$ with probability $p^{|A|}$.
By Chernoff bounds whp  $\bsize{\bigcap_{S \in A} G(S)}  = p^{|A|} n \pm (p^{|A|} n)^{0.6}$.
Chernoff also gives whp $|G| = p \tbinom{n}{r} \pm (pn^r)^{0.6}$. The lemma follows.
\end{proof}

One consequence is that if we choose the reserve $R$ randomly
then whp it satisfies Lemma \ref{lem:R}, namely it is suitably bounded and  for any edge $e \in K^r_n \sm R$ 
there are many cliques $Q_e$ as required in the cover step, with $e \in Q_e$ and $Q_e \sm \{e\} \sub R$.

\begin{proof}[Proof of Lemma \ref{lem:R}]
We choose  $R \sim K^r_n(n^{-\rho})$ randomly. As $\rho<1/3k$, by Lemma \ref{lem:randomtyp}
whp $R$ is $(n^{-0.1},k)$-typical with $d(R) = (1 \pm n^{-0.1}) n^{-\rho}$.
In particular, $R$ is $2n^{-\rho}$-bounded by the case $|A|=1$ of typicality.
Also, we can extend any $e \in K^r_n \sm R$ to a clique $Q_e$ with $Q_e \sm \{e\} \sub R$
by choosing $V(Q_e) \sm e$ one vertex at a time. By typicality, there are 
$(1 \pm n^{-0.1}) (n^{-\rho})^{\tbinom{r+i}{r}-\tbinom{r+i-1}{r}} n$ choices for the $i$th vertex.
Multiplying these choices for $1 \le i \le q-r$ and dividing by $(q-r)!$ for the order, there are
$(1 \pm qn^{-0.1}) (n^{-\rho})^{k-1} \tbinom{n}{q-r}$ choices for $Q_e$.
\end{proof}

Next we state a general purpose random greedy algorithm that captures
the various random greedy algorithms used earlier in the paper
(for covering, splitting and elimination). The accompanying lemma shows
that these random greedy algorithms are whp successful 
on suitably bounded inputs and produce suitably bounded outputs.
For simplicity, we state and prove it at first for extensions in the complete $r$-graph;
it will then be clear (as explained in the remark following the lemma)
how to modify it for extensions in a random $r$-graph (as needed for the cover step).

Here for any fixed $r$-graph $H$ and $F \sub V(H)$, an `extension' (of type) $(H,F)$ 
refers to an embedding of the induced subgraph $H[F] := \{e \in H: e \sub F\}$
which is to be extended to an embedding of $H$.

\begin{defn} \label{def:process}
Let $H$ be an $r$-graph and $F \sub V(H)$.
We call the extension of type $(H,F)$ admissible if for every edge $e \in H \sm H[F]$
there is some $f \in H[F]$ that contains $e \cap F$.
Let $B \sub K^r_n$ and $\Phi=(\phi_1,\dots,\phi_t)$ 
be a sequence of embeddings of $H[F]$ in $K^r_n$.
The $(\Phi,B)$-process is a random sequence $\Phi^*=(\phi^*_1,\dots,\phi^*_t)$
of embeddings of $H$ in $K^r_n$, where each $\phi^*_i$ restricts to $\phi_i$ on $H[F]$.
For each $i$, conditional on the previous embeddings $(\phi^*_j: j<i)$,
writing $C_i$ for the edges covered by all previous $\phi^*_j(H[F])$,
the process chooses $\phi^*_i$ uniformly at random subject to
not using any edge of $B$ or $C_i$, or aborts if no such choice is possible.
\end{defn}

For any edge $e$ of $H$ and $j \in [t]$ we write 
$\pl^j_e \Phi^* = \sum_{i \in [t]} 1_{\phi^*_i(e)} \in \mb{N}^{K^r_n}$
for the $r$-multigraph formed by all copies of $e$ in the process.
If $e \in H \sm H[F]$ then $\pl^j_e \Phi^*$ will be an $r$-graph (with no multiple edges).
For $e \in H[F]$ we note that $\pl^j_e \Phi^*$ is non-random, depending only on $\Phi$,
so we can denote it by $\pl^j_e \Phi$; here we need to allow multiple edges
for the application to Splitting in subsection \ref{sub:construct}.

\begin{lemma} \label{lem:process}
Consider the  $(\Phi,B)$-process as in Definition \ref{def:process},
assuming that $(H,F)$ is admissible,  
$B$ and all $\pl^t_e \Phi$ with $e \in H[F]$ are $\tT$-bounded, 
where $n^{-1/2} < \tT < (8r!^2 |H|)^{-1}$ 
and $n > n_0(v_H)$ is sufficiently large.
Then whp the process does not abort and all $\phi^t_e \Phi^*$ 
are $2^{r+1} r! \tT$-bounded.
\end{lemma}

\begin{proof}
For $i \in [t]$ let $\mc{B}_i$ be the bad event that 
some $\phi^i_e \Phi^*$ with $e \in H$ is not $2^{r+1} r! \tT$-bounded.
Define a stopping time $\tau$ to be the first $i$ for which $\mc{B}_i$ holds,
or $\infty$ if there is no such $i$. To prove the lemma,
it suffices to show whp $\tau = \infty$.

We fix $i' \in [t]$ and bound $\mb{P}(\tau=i')$ as follows.
For any $i<i'$, as $\mc{B}_i$ does not hold, all $\phi^i_e \Phi^*$ are  $2^{r+1} r! \tT$-bounded.
There are $(1+O(1/n))n^{v_H - |F|}$ possible choices of $\phi^*_i$,
of which at most $|H| \cdot  2^{r+1} r! \tT n^{v_H - |F|}$ are forbidden 
due to using some edge in $B$ or $\phi^i_e \Phi^*$ for some $e \in H$.
As $\tT < (8r!^2 |H|)^{-1}$,
at least $\tfrac12 n^{v_H - |F|}$ choices of $\phi^*_i$ are not forbidden.

Now we can bound $\mb{P}'(\phi^*_i(e)=g)$ 
for any $e \in H \sm H[F]$ and $g \in K^r_n$,
where $\mb{P}'$ denotes probability conditional
on the history of the process before step $i$.
The event $\{\phi^*_i(e)=g\}$ is only possible 
if $g \cap \phi_i(F) = \phi_i(e \cap F)$.
Assuming this, and writing $r'=|e \sm F| \ge 1$,
at most $r'! n^{v_H-|F|-r'}$ choices of $\phi^*$ give $\phi^*_i(e)=g$, 
so  $\mb{P}'(\phi^*_i(e)=g) \le 2r'! n^{-r'}$.

Next we fix $e$ and $g$ as above 
and bound $\sum_{i \le i'} \mb{P}'(\phi^*_i(e)=g)$.
As $(H,F)$ is admissible, we can fix
some $f \in H[F]$ that contains $e \cap F$.
There are at most $\tbinom{r}{r'} \tT n^{r'}$
choices of $i$ such that $g \cap \phi_i(F) = \phi_i(e \cap F)$,
as we can bound by fixing $g \cap \phi_i(F)$ in $\tbinom{r}{r'}$ ways
then extending to $\phi_i(f) \in \pl^t_f \Phi$, which is $\tT$-bounded.
Thus $\sum_{i \le i'} \mb{P}'(\phi^*_i(e)=g) \le 2r'!  \tbinom{r}{r'} \tT \le 2^r r! \tT$.
 
Finally, we can estimate the probability that $\phi^{i'}_e \Phi^*$ is not $2^{r+1} r! \tT$-bounded.
To do so, fix any $S \in \tbinom{[n]}{r-1}$ 
and let $X$ count edges of  $\phi^{i'}_e \Phi^*$ that contain $S$. 
By the previous calculation, we have $\mb{E}X \le 2^r r! \tT n$.
Also $X$ is stochastically dominated by a sum of independent random variables,
each of which is bounded by $v_H$, so whp $X < 2\mb{E}X$ by Chernoff bounds.
Taking a union bound over $S$ and $e$, whp $\tau>i'$, as required.
\end{proof}

\begin{rem} \label{rem:process}
The same proof and conclusion holds assuming that $B$ is $\tT_B$-bounded
with $n^{-1/2} < \tT_B < (8r!^2 |H|)^{-1}$ unrelated to $\tT$:
whp the process does not abort and all $\phi^t_e \Phi^*$ 
are $2^{r+1} r! \tT$-bounded.
We also have a similar result for more general processes where
each $\phi^*_i$ is chosen from some prescribed set $\mc{H}_i$
of embeddings of $H$ restricting to $\phi_i$ on $H[F]$,
for example, those with $\phi_i(H \sm H[F]) \sub K$
for some given subgraph $K$ of $K^r_n$.
If each $|\mc{H}_i| \ge \oO n^{v_H - |F|}$ for some $\oO>0$
and $n^{-1/2} < \tT <  \oO (8r!^2 |H|)^{-1}$ 
then whp the process does not abort and all $\phi^t_e \Phi^*$ 
are $2^{r+1} r! \oO^{-1} \tT$-bounded.
\end{rem}

We conclude by applying the above general result
to  the particular random greedy algorithms used by the algorithm,
including that needed for covering edges by cliques in Lemma \ref{lem:cover}.
Here we note that $(\OO,V(\hQ^+))$ and $(\OO,V(\hQ^+) \cup V(\hQ^-))$
are admissible types of extension by Lemma \ref{lem:OO}.

\begin{cor} \label{cor:A}
The random greedy algorithms in subsection \ref{sub:construct} are whp successful,
producing $A$ that is $n^{-\aA/4}$-bounded. 
\end{cor}

\begin{proof}[Proof of Corollary \ref{cor:A} and Lemma \ref{lem:cover}]
After Step 1 we have $\mc{Q}_1 \sub K^r_q(K^r_n)$
such that $\bigcup \mc{Q}_1 \sub \pl \mc{Q}_1$ is $n^{-\aA/2}$-bounded.

In Step 2, we consider admissible extensions of type $(K^r_{r+q},e_0)$
for some edge $e_0$ of $K^r_q$ and the $(\Phi,B)$-process
with $B=\bigcup \mc{Q}_1$ and $\Phi$ a sequence of bijections 
of $e_0$ with edges of $\bigcup \mc{Q}_1$. 
By Lemma \ref{lem:process} with $\tT = n^{-\aA/2}$ 
whp the algorithm is successful and produces $\mc{Q}_2$
such that $A_0 = \bigcup( \mc{Q}_1 \cup  \mc{Q}_2)$ is $2^{r+1} r! \tbinom{q+r}{r} n^{-\aA/2}$-bounded. 
Similarly, in Lemma \ref{lem:cover} we consider $(K^r_q,e_0)$
and the $(\Phi,\es)$-process where $\Phi$ corresponds to some $n^{-3k\rho}$-bounded $L_1 \sub K^r_n \sm R$, 
where now the clique $Q_e$ chosen for some edge $e$ must also satisfy $Q_e \sm \{e\} \sub R$.
By Lemma \ref{lem:R} we can apply Remark \ref{rem:process} 
with $\oO = n^{-k\rho}$ and $\tT = n^{-3k\rho}$ 
to see that whp this process produces cliques $Q_e$ as required to prove Lemma \ref{lem:cover}.

In Step 3, we consider admissible extensions $(\OO,V(\hQ^+))$
in the $(\Phi,B)$-process with $B=A_0$ 
and $\Phi$ corresponding to  $Q_1,\dots,Q_t$
consisting of $2^{q+1} r!$ copies of each clique in $\mc{Q}_1 \cup \mc{Q}_2$,
so each $\pl^t_e \Phi$ with $e \in \hQ$ is $2^{q+r+2} r!^2   \tbinom{q+r}{r}  n^{-\aA/2}$-bounded. 
We note that every edge is covered by  at most $2^{q+1}r!(k+2)$ cliques in $Q_1,\dots,Q_t$,
as every edge is covered at most twice by $\mc{Q}_1$ 
and at most $\tbinom{q}{q-r}=k$ times by $\mc{Q}_2$.
We slightly restrict the process by forbidding any $\OO_i$ on $Q_i$
such that $V(\OO_i) \sm V(Q_i)$ intersects some $V(\OO_j) \sm V(Q_j)$
where $Q_i$ and $Q_j$ share an edge; this forbids at most 
$2^{q+1}r!(k+2) \cdot qv_\OO n^{v_\OO-q-1} < 0.1 n^{v_\OO-q}$ choices of $\OO_i$. 
By Lemma \ref{lem:process} and Remark \ref{rem:process},
whp the algorithm produces some $\mc{Q}_3$
such that $\bigcup_{i=1}^3 \bigcup \mc{Q}_i$ is $2^{q+2r+4} (2q)^{4r} n^{-\aA/2}$-bounded,
using $|\OO| \le 3(2q)^r k^2$.

In Step 4, we consider admissible extensions $(\OO,V(\hQ^+) \cup V(\hQ^-))$
in the $(\Phi,B)$-process with $B=\bigcup_{i=1}^3 \bigcup \mc{Q}_i$ 
and $\Phi$ corresponding to  $Q^\pm_1,\dots,Q^\pm_{t'}$,
which is a sequence in which each previous clique appears at most $2^{q+1}r!(k+2)$ times.
Thus each $\pl^t_e \Phi$ with $e \in \hQ^+ \cup \hQ^-$ 
is $2^{2q+2r+6} (2q)^{5r}  n^{-\aA/2}$-bounded,
so by Lemma \ref{lem:process} whp the algorithm produces some $\mc{Q}_4$
such that $\bigcup_{i=1}^4 \bigcup \mc{Q}_i$ is $2^{2q+3r+7} r! (2q)^{5r}  n^{-\aA/2}$-bounded.

Step 5 uses the same extensions as Step 4 and is similarly whp successful,
producing a final set $\mc{Q}$ such that $\bigcup \mc{Q}$ 
is  $2^{2q+4r+8} r!^2 (2q)^{5r} n^{-\aA/2}$-bounded, so $n^{-\aA/4}$-bounded for $n>n_0$.
\end{proof}

We have now completed the main step
in the proof of Theorem \ref{thm:steiner};
the other steps can be found in various other papers,
but for completeness we will devote the remainder of the paper to their presentation, 
with some simplifications and quantitative improvements.

\section{The integral absorber}  \label{sec:int}

Here we construct the integral absorber
used in Step 1 of subsection \ref{sub:construct}.
The original construction can be found in \cite[Section 5]{K}.
Here we will adapt the argument in  \cite[Section 5]{KeSS},
also adding a new lemma for controlling edge multiplicities.
The resulting proof is not shorter than the original, but the method is more widely applicable, 
and will be simpler for us to present in this paper, as it only uses methods that we have
already discussed above (we will not need to discuss the characterisation of the 
decomposition lattice via  `octahedral decompositions');
it also serves to illustrate the more complicated argument in \cite{KeSS}.

\begin{lemma} \label{lem:Aint}
Let $q,r,\rho,\aA,n$ be as in \eqref{param}
and suppose $R \sub K^r_n$ is $n^{-\rho}$-bounded.
Then there is $\mc{Q}_1 \sub K^r_q(K^r_n)$ such that 
$\pl \mc{Q}_1$ is $n^{-\aA/2}$-bounded, 
every edge of $K^r_n$ is in at most two cliques of $\mc{Q}_1$,
and for any $K^r_q$-divisible $J$ supported in $R$
there is $\Phi \in \mb{Z}^{\mc{Q}_1}$ with $\pl \Phi = J$.
\end{lemma}

\subsection{Overview}
 
To introduce our strategy for proving Lemma \ref{lem:Aint},
we consider a simpler first attempt and then discuss how this attempt
will be modified in our actual strategy.

Consider some sparse random $K \sub K^r_n \sm R$. 
By Lemma \ref{lem:randomtyp}, whp for any $e \in R$ 
there are many cliques $Q_e$ that contain $e$ and have all other edges in $K$. 
Include in $\mc{Q}_1$ one such clique for each $e \in R$.
Given any  $J$ as above, we can then obtain $J_1 = J - \pl \Phi_1$ 
supported in $K$ for some $\Phi_1 \in \mb{Z}^{\mc{Q}_1}$,
thus reducing the problem to generating such $J_1$.

For any such $J_1$, as $J_1$ is $K^r_q$-divisible
we can write $J_1 = \pl \Phi$ for some $\Phi \in \mb{Z}^{K^r_q(K^r_n)}$.
Using clique exchanges as in Step 3 (Splitting) above,
we can replace $\Phi$ by some $\Phi'$ with $\pl \Phi' = \pl \Phi = J_1$
such that $\Phi'$ is supported only on cliques 
with at most one edge outside of $K$.
The cliques using edges outside of $K$
can be partitioned into cancelling pairs
and eliminated as in Step 4 (Elimination) above,
thus replacing $\Phi'$ by some $\Phi''$ with $\pl \Phi'' = \pl \Phi' = J_1$
such that $\Phi''$ is supported on $K^r_q(K)$.
This reduces the problem to generating all $\pl Q$ with $Q \in K^r_q(K)$.

It is not clear how to make this attempt work, so we modify it as follows.
Starting with $K$ as above,  we will delete a small fraction of its edges
to pass to a subgraph $K^*$ for which 
we can find a suitably bounded set $\mc{G}$ of `generating' $q$-cliques
and a small set $\mc{S}$ of `saturated' $q$-cliques such that
(a) every edge of $K^*$ is in many unsaturated cliques, and
(b) the unsaturated cliques are generated over $\mb{Z}/N\mb{Z}$ by $\mc{G}$.

Next we sample independently random permutations $\sS_1,\dots,\sS_u$
of the vertex set $[n]$, where $u$ is a suitably large constant.
We consider all the `rotated' $r$-graphs $\sS_i(K^*) \sub K^r_n$,
which we think of as `colours', noting that any edge of $K^r_n$ 
may be uncoloured or may have several colours.
Then $\mc{Q}_0 := \bigcup_{i=1}^u \sS_i(\mc{G})$
generates  all `monochromatic' unsaturated cliques
$\sS_i(Q)$ with $Q \in K^r_q(K) \sm \mc{S}$.

We then obtain $\mc{Q}'_0$ by adding cliques $Q_e$ for $e \in R$ as in the first attempt above
and similarly to Step 2 of subsection \ref{sub:construct} adding cliques for local decoding,
to account for our use of $\mb{Z}/N\mb{Z}$ rather than $\mb{Z}$ in (b) above.
A further `flattening' step will replace  $\mc{Q}'_0$ by $\mc{Q}_1$ that uses any edge at most twice.

In our analogue of the above reduction to $\pl \Phi'' = J_1$
we will ensure that $\Phi''$ is supported on cliques that are `rainbow', 
meaning that each edge belongs to a different colour $\sS_i(K^*)$.
It will then suffice to show that rainbow cliques 
are generated by monochromatic unsaturated cliques,
which we achieve by finding clique exchanges on any rainbow base clique 
in which all non-base cliques in the two decompositions are monochromatic
and the colours of these cliques are distinct
(thus each non-base edge will belong to one positive monochromatic clique
and one negative monochromatic clique, which will have different colours).

The parts of the proof in the previous paragraph are quite subtle, 
as they rely on second moment calculations
showing that the failure probability for finding suitable configurations
is polynomially small, although apparently not small enough for a union bound.
Fortunately, this is good enough, as we can boost the probability by considering 
several disjoint sets of colours, using independence of the random permutations;
this argument dictated our choice of reserve and absorber density with polynomial decay,
so we also require (and can obtain) such bounds in the nibble.

\subsection{Generating and saturated cliques}

Similarly to our notation for integral characteristic vectors in subsection \ref{sub:motiv},
we now consider modular characteristic vectors over $\GG := \mb{Z}/N\mb{Z}$:
for $\Phi \in \GG^{K^r_q(K^r_n)}$ we define $\pl \Phi \in \GG^{K^r_n}$
by $(\pl \Phi)_e = \sum \{ \Phi_Q: e \in Q \in K^r_q(K^r_n) \}$.

For a subset $S$ of an abelian group
we write $\bgen{S}$ for the subgroup generated by $S$.
In particular, the subgroup $\bgen{\pl Q: Q \in K^r_q(K^r_n)}$ of $\GG^{K^r_n}$
plays the role of $K^r_q$-divisible vectors in the integral setting.

Unless it is clear from the context, 
we use $\bgen{\cdot}_\GG$ or  $\bgen{\cdot}_{\mb{Z}}$ 
to indicate spans over $\GG$ or $\mb{Z}$.

\begin{lemma} \label{lem:KSG}
For $q,r,\aA,n$ as in \eqref{param},
if $K \sim K^r_n(n^{-\aA})$ is $n^{-\aA}$-random in $K^r_n$ 
then whp $K$ is $(n^{-0.1},|\OO|)$-typical and there are 
\begin{enumerate}
\item a `good' subgraph $K^* \sub K$ with $|K \sm K^*| < n^{-0.1\aA} |K|$,
\item a set  $\mc{S} \sub K^r_q(K)$ of `saturated' cliques 
with $|\mc{S}| < n^{-0.1\aA-k\aA} \tbinom{n}{q}$
such that every edge in $K^*$ is in
$(1 \pm n^{-0.1\aA})  n^{\aA-k\aA} \tbinom{n}{q-r}$ 
unsaturated $q$-cliques in $K^r_q(K) \sm \mc{S}$, and
\item a set $\mc{G} \sub K^r_q(K)$ of `generating' cliques 
such that  $\mc{G}$ generates $K^r_q(K) \sm \mc{S}$ over $\GG$
and $\pl \mc{G}$ is $2^q n^{-0.7\aA}$-bounded.
\end{enumerate}
\end{lemma}

\begin{proof}
Typicality of $K$ holds whp by Lemma \ref{lem:randomtyp}.
We consider the following greedy algorithm for constructing $\mc{G}$.
We start with $\mc{G}=\es$.
At each step  we call an $(r-1)$-set $f \in \tbinom{[n]}{r-1}$ saturated 
   if it is contained in at least $n^{1-0.7\aA}$ cliques in $\mc{G}$,
   and we call a clique $Q \in K^r_q(K)$ saturated if it contains some saturated $(r-1)$-set.
If there is some unsaturated $Q \in K^r_q(K)$ with $\pl Q \notin \bgen{\mc{G}}_\GG$ 
 then we add some such $Q$ to $\mc{G}$ and proceed to the next step, otherwise we stop.
The algorithm terminates with some sets $\mc{S} $ of `saturated' cliques and
$\mc{G}$ of `generating' cliques such that $\pl \mc{G}$ is $2^q n^{-0.7\aA}$-bounded
and any $Q \in K^r_q(K)$ with $\pl Q \notin \bgen{\mc{G}}_\GG$ is saturated,
so $\mc{G}$ generates $K^r_q(K) \sm \mc{S}$ over $\GG$.

We note that $|\mc{G}| \le N|K|$, 
as each new element increases the subgroup $\bgen{\mc{G}}_\GG \sub \GG^K$,
and any chain of subgroups in $\GG^K$ has length at most $N|K|$ (or even $|K|\log_2 N$).
Thus there are at most 
 $\tbinom{q}{r-1}N|K|/n^{1-0.7\aA}  < 2^q N n^{-0.3\aA} \tbinom{n}{r-1}$ saturated $(r-1)$-sets.
By typicality of $K$, we deduce
$|\mc{S}| \le 2^q N n^{-0.3\aA} \tbinom{n}{r-1} \cdot  2 n^{-k\aA} \tbinom{n}{q-r+1}
   <  2^{3q} N n^{-0.3\aA} n^{-k\aA} \tbinom{n}{q}$.

We define $K^* = K \sm K_0$, where $K_0$ is the set of edges $e$ contained in at least 
$n^{0.85\aA} n^{-k\aA} \tbinom{n}{q-r}$ cliques in $\mc{S}$.
Double-counting pairs $(e,Q)$ with $e \in Q \cap K_0$ and $Q \in \mc{S}$, we have
$|K_0| \cdot n^{0.85\aA} n^{-k\aA} \tbinom{n}{q-r} \le |\mc{S}| \tbinom{q}{r}$,
so $|K_0| \le 2^{4q} N n^{-1.15\aA} \tbinom{n}{r} < n^{-0.1\aA} |K|$, as $n > (2^{4q} N)^{20/\aA}$.
By typicality of $K$, any edge of $K^*$ is in 
$(1 \pm n^{-0.1})  n^{\aA-k\aA} \tbinom{n}{q-r}$ cliques of $K^r_q(K)$,
so in $(1 \pm n^{-0.1\aA})  n^{\aA-k\aA} \tbinom{n}{q-r}$ of $K^r_q(K) \sm \mc{S}$
by definition of $K_0$. This completes the proof.
\end{proof}

Henceforth we fix $K, K^*, \mc{S}, \mc{G}$ as in the previous lemma.

 \subsection{Colours}

Next we sample independently random permutations $\sS_1,\dots,\sS_u$
of the vertex set $[n]$ with $u := 20q^2 \aA^{-1} |\OO|$.
We consider all the `rotated' $r$-graphs $\sS_i(K^*)  \sub K^r_n$,
which we think of as `colours' and let $\mc{Q}_0 := \bigcup_{i=1}^u \sS_i(\mc{G})$.

By Lemma \ref{lem:KSG},
$\pl \mc{Q}_0$ is $u 2^q n^{-0.7\aA}$-bounded,
and $\mc{Q}_0$ generates over $\GG$ all `monochromatic' unsaturated cliques
$\sS_i(Q)$ with $Q \in K^r_q(K) \sm \mc{S}$.
We fix  $\sS_1,\dots,\sS_u$ so that the following properties hold.

\begin{lemma} \label{lem:extcol}
With failure probability at most $n^{-1}$ under the random choice of $\sS_1,\dots,\sS_u$:
\begin{enumerate}[label=\arabic*.]
\item For any edge $e' \in K^r_n$ there are at least $\tfrac12 (n^{-\aA})^{k-1} n^{q-r}$
cliques $Q$ containing $e'$ such that $Q \sm \{e'\}$ is rainbow,  meaning that there are distinct colours
$(i_e: e \in Q \sm \{e'\})$ such that each $e \in \sS_{i_e}(K^*)$.
\item For any clique $Q \in K^r_q(K^r_n)$ there is some copy $\phi(\OO)$ of $\OO$ with $\phi(\hQ^+)=Q$
such that $\phi(\OO) \sm Q$ is rainbow, meaning that there are distinct colours
$(i_e: e \in \OO \sm \hQ^+)$ such that each $\phi(e) \in \sS_{i_e}(K^*)$.
\item For any two cliques $Q^\pm \in K^r_q(K^r_n)$ intersecting in exactly one edge $e$
there is some copy $\phi(\OO)$ of $\OO$ with $\phi(\hQ^\pm)=Q^\pm$
such that $\phi(\OO) \sm (Q^+ \cup Q^-)$ is rainbow.
\item Rainbow cliques are generated over $\GG$ by monochromatic unsaturated cliques:
$\pl Q \in \bgen{\mc{Q}_0}_{\GG}$ for any $Q \in K^r_q(K^r_n)$ such that
there is an injection $\psi:Q \to [u]$ with $e \in \sS_{\psi(e)}(K^*)$ for all $e \in Q$.
\end{enumerate}
\end{lemma}

\begin{proof}
For (1), fix any $e_0 \in K^r_q$ and distinct colours $i_e$ for each $e \in K^r_q \sm \{e_0\}$.
We consider the random variable $X$, depending on $\sS_{i_e}$ for each such $e$,
counting cliques $Q = \phi(K^r_q) \in K^r_q(K^r_n)$ with $\phi(e_0)=e'$
that are rainbow of the appropriate colour,
meaning that the events $E^\phi_e = \{ \phi(e) \in \sS_{i_e}(K^*)\}$ all occur. By Lemma \ref{lem:KSG}.i
each  $\mb{P}(E^\phi_e) = |K^*|/\tbinom{n}{r} = (1 \pm 2n^{-0.1\aA}) n^{-\aA}$ independently,
so $\mb{E}X = ( (1 \pm 2n^{-0.1\aA}) n^{-\aA} )^{k-1} \tbinom{n}{q-r}$.

To estimate the second moment of $X$, we note that $X^2$ counts pairs $(\phi,\phi')$ 
for which all $E^\phi_e$ and $E^{\phi'}_e$ occur. 
There are at most $q^2 n^{2(q-r)-1}$ such pairs in which 
$\phi(K^r_q)$ and $\phi'(K^r_q)$ share a vertex not in $e$.
For any other pair $(\phi,\phi')$ we note that each
pair $(f,f') := (\sS_{i_e}^{-1}(\phi(e)), \sS_{i_e}^{-1}(\phi'(e)))$
is uniformly distributed over all $(f,f') \in K^r_n \times K^r_n$ 
with $|f \cap f'|=|e \cap e_0|$. By typicality of $K$, we have 
$\mb{P}(E^\phi_e \cap E^{\phi'}_e) \le \mb{P}(f,f' \in K) \le (1 + n^{-0.1}) n^{-2\aA}$.
By independence of these events as $e$ varies, we deduce 
$\mb{E}X^2 \le ((1 + n^{-0.1}) n^{-2\aA})^{k-1} \tbinom{n}{q-r}^2 + q^2 n^{2(q-r)-1}$.

Thus $X$ has variance at most $k n^{-0.1\aA} (\mb{E}X)^2$,
so by Chebyshev's inequality $\mb{P}(X \le 0.6\mb{E}X) \le 8kn^{-0.1\aA}$.
This bound is not yet good enough to take a union bound over $e$,
so the final observation is that as $u > 20rk \aA^{-1}$ we can apply the above argument
to $20r \aA^{-1}$ disjoint sets of colours $(i_e: e \in K^r_q \sm \{e_0\})$
and multiply the corresponding bounds for $\mb{P}(X \le 0.6\mb{E}X)$ by independence.
Taking a union bound over $e' \in K^r_n$ establishes (1).

The proofs of (2) and (3) are very similar 
so we just describe the necessary modifications.
For (2) we replace $(K^r_q,e_0)$ by $(\OO,V(\hQ^+))$,
fix distinct colours $(i_e: e \in \OO \sm \hQ^+)$
and let $X$ count embeddings $\phi$ of $\OO$
with $\phi(\hQ^+)=Q$ and each $\phi(e) \in \sS_{i_e}(K^*)$.
We calculate $\mb{E}X = ( (1 \pm n^{-0.1\aA}) n^{-\aA} )^{|\OO \sm \hQ^+|} n^{v_\OO - q}$
and $\mb{E}X^2 \le (1 + |\OO| n^{-0.1\aA}) (\mb{E}X)^2$, so we conclude similarly to (1).
For (3) we consider the extension $(\OO,V(\hQ^+) \cup V(\hQ^-))$,
fix distinct colours $(i_e: e \in \OO \sm (\hQ^+ \cup \hQ^-))$
and  let $X$ count embeddings $\phi$ of $\OO$
with $\phi(\hQ^\pm)=Q^\pm$ and each $\phi(e) \in \sS_{i_e}(K^*)$.
We calculate $\mb{E}X = ( (1 \pm n^{-0.1\aA}) n^{-\aA} )^{|\OO \sm (\hQ^+ \cup \hQ^-)|} n^{v_{\OO} - 2q+r}$
and $\mb{E}X^2 \le (1 + |\OO| n^{-0.1\aA}) (\mb{E}X)^2$, so we conclude similarly to (1),
as $u = 20q^2 \aA^{-1}|\OO|$ so we can use $20q^2 \aA^{-1}$ disjoint sets of colours.

For (4), we require all of the properties of $\OO$ obtained in Lemma \ref{lem:OO}:
we have  two $K^r_q$-decompositions $\Ups^\pm$ and cliques
$\hQ^+ \in \Ups^+$ and $\{ \hQ^e : e \in \hQ^+ \} \sub \Ups^-$
such that each $\hQ^e \cap \hQ^+ = \{e\}$ and
$\{ V(\hQ^e) \sm V(\hQ^+) : e \in \hQ^+\}$ are pairwise disjoint,
and the extension $(\OO,F)$ with $F := V(\hQ^+) \cup \bigcup_e V(\hQ^e)$ is admissible,
meaning that any edge of $\OO$ intersects $F$ in a subset of $\hQ^+$ or some $\hQ^e$.

Given the colours $\psi(e)$ of edges $e$ in the base clique $Q$,
we condition on all $\sS_{\psi(e)}$ and fix any embedding $\phi_F$ of $\OO[F]$
with $\phi_F(\hQ^+)=Q$ and each $\phi_F(\hQ^e) \in \sS_{\psi(e)}(K^r_q(K) \sm \mc{S})$
being a monochromatic unsaturated clique of the appropriate colour.
By admissibility and Lemma \ref{lem:KSG}.ii there are 
$(1+O(n^{-1})) \brac{ (1 \pm n^{-0.1\aA})  n^{\aA(1-k)} \tbinom{n}{q-r} }^{k}$
choices of $\phi_F$. We fix new distinct colours $i_{\hQ'}$ for each clique $\hQ' \in \Ups^\pm$
not equal to $\hQ^+$ or some $\hQ^e$, and consider the random variable $X$,
depending on the permutations $\sS_{i_{\hQ'}}$ for such $\hQ'$,
counting extensions of $\phi_F$ to some embedding $\phi$ of $\OO$
such that each $\phi(\hQ')$ is a monochromatic unsaturated clique of the appropriate colour,
meaning that the events $E^\phi_{\hQ'} = \{ \phi(\hQ') \in \sS_{i_{\hQ'}}(K^r_q(K) \sm \mc{S}) \}$ all occur.
By typicality of $K$ and Lemma \ref{lem:KSG}.ii
each $\mb{P}(E^\phi_{\hQ'}) = |K^r_q(K) \sm \mc{S}|/\tbinom{n}{q} = (1 \pm 2n^{-0.1\aA}) n^{-k\aA}$ independently,
so $\mb{E}X = ((1 \pm 2n^{-0.1\aA}) n^{-k\aA})^M n^{v_{\OO}-|F|}$, where there are $M < 2|\OO|/k$ such $\hQ'$.

To estimate $\mb{E}X^2$, we can focus on pairs of embeddings $(\phi,\phi')$ that are vertex-disjoint
outside of $\phi_F(F)$, as other pairs only contribute at most $n^{2(v_{\OO}-|F|)-1}$.
For such $(\phi,\phi')$, each $(\sS_{i_{\hQ'}}^{-1}(\phi(\hQ')), \sS_{i_{\hQ'}}^{-1}(\phi'(\hQ')))$
is uniformly distributed over all $(Q,Q') \in K^r_q(K^r_n)^2$ with 
$|V(Q) \cap V(Q')|=|V(\hQ') \cap F|$.  By typicality of $K$, 
if $|V(\hQ') \cap F|<r$ we have 
$\mb{P}(E^\phi_{\hQ'} \cap E^{\phi'}_{\hQ'}) \le \mb{P}(Q,Q' \in K^r_q(K))
\le (1 + n^{-0.1}) (n^{-\aA})^{2k}$,
or if $|V(\hQ') \cap F|=r$, letting $e' = \phi_F(\hQ'[F]) \in \phi_F(\OO[F])$, we have
$\mb{P}(E^\phi_{\hQ'} \cap E^{\phi'}_{\hQ'}) \le \mb{P}( (Q \cup Q') \sm \{e'\} \sub K)
\le (1 + n^{-0.1}) (n^{-\aA})^{2(k-1)}$.
By independence as $\hQ'$ varies, we deduce 
$\mb{E}X^2 \le (1 +  8|\OO| n^{-0.1\aA}) (\mb{E} X)^2$.
By Chebyshev's inequality and using independence for $20q \aA^{-1}$ disjoint sets
of colours $(i_{\hQ'})$, we can take a union bound to ensure that the
required copy of $\OO$ exists for any base clique $Q$.
\end{proof}

\subsection{Focussing, decoding, flattening}

Here we complete the construction of the integral absorber $\mc{Q}_1$.
The first lemma augments the above set $\mc{Q}_0$ of $q$-cliques,
forming a new set $\mc{Q}'_0$ 
that incorporates local decoders as in Step 2 of subsection \ref{sub:construct}
and also cliques with exactly one edge in $R$,
allowing any vector in $R$ to be focussed 
into the coloured $r$-graph $\bigcup_{i=1}^u \sS_i(K^*)$.
The second lemma flattens $\mc{Q}'_0$ to produce $\mc{Q}_1$ 
with $\bgen{\mc{Q}'_0}_{\mb{Z}} \sub \bgen{\mc{Q}_1}_{\mb{Z}}$ 
such that any edge is used by at most two cliques of $\mc{Q}_1$.

\begin{lemma} \label{lem:Q0'}
Let  $\mc{Q}_0 \sub K^r_q(K^r_n)$ such that $\pl \mc{Q}_0$ is $u 2^q n^{-0.7\aA}$-bounded.
Then there is $\mc{Q}_0 \sub \mc{Q}'_0  \sub K^r_q(K^r_n)$ 
such that $\pl \mc{Q}'_0$ is $2^{q+2} (4q)^r u n^{-0.7\aA}$-bounded,
each $e \in \bigcup \mc{Q}_0$ is contained in some `locally decoding'  
$(q+r)$-set $Z_e$ with $K^r_q( \tbinom{Z_e}{r} ) \sub \mc{Q}'_0$,
and each $e \in R$ is contained in some `focussing clique' $Q_e \in \mc{Q}'_0$
with $Q_e \sm \{e\} \sub \bigcup_{i=1}^u \sS_i(K^*)$.
\end{lemma} 

\begin{proof}
We choose the required sets $Z_e$ and $Q_e$ randomly.
The proof would work with independent uniformly random choices,
but as we have already set up the machinery
it will be more convenient to use random greedy choices.
Thus we apply two random greedy algorithm to choose the required
$Z_e$ for  $e \in \bigcup \mc{Q}_0$ and $Q_e$ for $e \in R$
so that each $r$-graph $\tbinom{Z_e}{r} \sm \{e\}$ and $Q_e \sm \{e\}$
is disjoint from all others and from $\bigcup \mc{Q}_0$.
The choice of $Z_e$ uses Lemma \ref{lem:process}
for the extension $(K^r_{r+q},e_0)$ as in the analysis of Step 2
in the proof of Corollary \ref{cor:A}, whp producing some set of cliques
covering some $r$-graph $B$ that is $2^{q+1} (4q)^r u n^{-0.9\aA}$-bounded.
For the choice of $Q_e$, by Lemmas \ref{lem:R} and \ref{lem:extcol}.1 
we can apply Remark \ref{rem:process} 
with $\tT = n^{-\rho}$, $\tT_B = 2^{q+1} (4q)^r u n^{-0.7\aA}$ and $\oO = \tfrac12 n^{-k\aA}$,
recalling from \eqref{param} that $\rho=(6k)^{-2}$ and  $\aA = (2q)^{-r} \rho$.
We whp obtain a $\tT'$-bounded $r$-graph $\bigcup_e \{Q_e\}$ 
with $\tT' < 2^{r+1} r! \oO^{-1} \tT < n^{-0.7\aA}$, so the final $\pl \mc{Q}'_0$ 
is $2^{q+2} (4q)^r u n^{-0.7\aA}$-bounded.
\end{proof}

The construction of $\mc{Q}_1$ is completed by the following flattening lemma,
which uses two greedy algorithms: a splitting algorithm so that bad cliques can
be unambiguously grouped according to their bad edge, then several rounds
of an elimination algorithm that repeatedly reduces edge multiplicities 
at the rate $x \mapsto \sqrt{x}$, so that we have at most $2\log \log n$ iterations
and incur only a polylogarithmic loss in boundedness.

\begin{lemma} \label{lem:flat}
Let  $\mc{Q}'_0 \sub K^r_q(K^r_n)$ such that 
$\pl \mc{Q}'_0$ is $2^{q+2} (4q)^r u n^{-0.7\aA}$-bounded.
Then there is $\mc{Q}_1 \sub K^r_q(K^r_n)$ 
with $\bgen{\mc{Q}'_0}_{\mb{Z}} \sub \bgen{\mc{Q}_1}_{\mb{Z}}$ 
such that $\pl \mc{Q}_1$ is $n^{-\aA/2}$-bounded, 
and any edge is used by at most two cliques of $\mc{Q}_1$.
\end{lemma} 

\begin{proof}
We start with a random greedy algorithm choosing 
for each $Q \in \mc{Q}'_0$ a clique exchange $\phi_Q(\OO)$
as in Step 3 (Splitting) of subsection \ref{sub:construct},
with each $Q = \phi(\hQ^+)$ and $\phi_Q(\OO) \sm Q$ edge-disjoint
from all others and from $\bigcup  \mc{Q}'_0$.
We obtain $\mc{Q}''_0$ from $\mc{Q}'_0$ by replacing each $Q$
by $\phi(\Ups^-) \cup \phi(\Ups^+ \sm \{\hQ^+\})$, 
so that $\bgen{\mc{Q}'_0} \sub \bgen{\mc{Q}''_0}$
as $\Ups^\pm$ are two $K^r_q$-decompositions of $\OO$.
By Lemma \ref{lem:process},
whp the process does not abort and  
$\pl \mc{Q}''_0$ is $n^{-0.6\aA}$-bounded, 
as $n_0 > (2^{5q} (4q)^r u)^{10/\aA}$.

Every edge not in $\bigcup  \mc{Q}'_0$
is used by at most two cliques of $\mc{Q}''_0$, 
and any clique of $\mc{Q}''_0$ 
has at most one edge in $\bigcup  \mc{Q}'_0$.

Now we set $\mc{Q}_1 = \mc{Q}''_0$ and then modify it
according the following random greedy algorithm.
At the start of each round, we let $Z \sub K^r_n \sm \bigcup \mc{Q}'_0$
be the set of edges used by more than two cliques of $\mc{Q}_1$.
If $Z=\es$ we stop, otherwise
for each $e \in Z$, in say $x$ cliques of $\mc{Q}_1$,
we divide the cliques containing $e$ into $\ell := \bfl{\sqrt{x}}$ groups
$C_1,\dots,C_\ell$ of sizes $\ell$ or $\ell+1$;
by construction of $\mc{Q}''_0$ no clique is considered by more than one $e$.
For each group $C_i$, we choose some $Q_i \in K^r_n$ that intersects
each $Q' \in C_i$ precisely in $e$ and is edge-disjoint from all
previously chosen cliques. For each other $Q' \in C_i$
we choose an elimination configuration $\phi_{Q'}(\OO)$ 
with $\phi_{Q'}(\hQ^+) = Q_i$ and  $\phi_{Q'}(\hQ^-) = Q'$,
so that each $\phi_{Q'}(\OO) \sm (Q_i \cup Q')$ 
is edge-disjoint from all others and from all previously chosen cliques.
We remove each $Q'$ from $\mc{Q}_1$, 
 and add  to $\mc{Q}_1$ each $Q_i$ and all $\phi_{Q'}(Q'')$
with $Q'' \in \Ups^\pm \sm \{\hQ^-,\hQ^+\}$, then proceed to the next round.

Recalling that $\Ups^\pm$ are two $K^r_q$-decompositions of $\OO$, we see that
each round maintains the property $\bgen{\mc{Q}'_0} \sub \bgen{\mc{Q}_1}$.
By Lemma \ref{lem:process},
at the start of any round when $\pl \mc{Q}_1$ 
is $\tT$-bounded with $\tT < (8r!^2 |\OO|)^{-1}$ 
whp the round does not abort and we obtain a new set $\mc{Q}_1$
such that $\pl \mc{Q}_1$ is $|\OO| \cdot 2^{r+1} r! \tT$-bounded.
Furthermore, if at the start of the round the maximum edge multiplicity is $x$
then at the end of the round it is at most $\max\{\bfl{\sqrt{x}}+1,2\}$.
The process terminates in at most 
$R = \bcl{\log_2 \log_2 |\pl \mc{Q}''_0|} < 2\log \log n$ rounds
at some $\mc{Q}_1$ with maximum edge multiplicity two 
such that $\pl \mc{Q}_1$ is $\tT$-bounded with
$\tT < (2^{r+1} r! |\OO|)^R \cdot n^{-0.6\aA} < n^{-\aA/2}$, 
for $n > (2q)^{90q\log q}$, say.
\end{proof}

\subsection{Completing the proof}

We conclude the section
by showing that the above construction of  $\mc{Q}_1$ is an integral absorber.

\begin{proof}[Proof of  Lemma \ref{lem:Aint}]
Let $q,r,\rho,\aA,n$ be as in \eqref{param}
and suppose $R \sub K^r_n$ is $n^{-\rho}$-bounded.
Let  $K, K^*, \mc{S}, \mc{G}$ be as in Lemma \ref{lem:KSG},
let $\mc{Q}_0 := \bigcup_{i=1}^u \sS_i(\mc{G})$ with
$\sS_1,\dots,\sS_u$ as in Lemma \ref{lem:extcol},
let $\mc{Q}'_0$ and $\mc{Q}_1$ be as in Lemmas \ref{lem:Q0'} and \ref{lem:flat}.

Consider any $K^r_q$-divisible $J$ supported in $R$.
We will find  $\Phi \in \mb{Z}^{\mc{Q}'_0}$ with $\pl \Phi = J$.
This suffices, as $\bgen{\mc{Q}'_0}_{\mb{Z}} \sub \bgen{\mc{Q}_1}_{\mb{Z}}$ 
by Lemma \ref{lem:flat}.
By Lemma \ref{lem:Q0'}, every $e \in R$ is in some `focussing clique'
$Q_e \in \mc{Q}'_0$ with $Q_e \sm \{e\} \sub \bigcup_{i=1}^u \sS_i(K^*)$.
Using these cliques we can fix $\Phi^0 \in \mb{Z}^{\mc{Q}'_0}$
such that $J^1 := J - \pl \Phi^0$ is supported in $\bigcup_{i=1}^u \sS_i(K^*)$.
As $J^1$ is  $K^r_q$-divisible there is
$\Phi^1 \in \mb{Z}^{K^r_q(K^r_n)}$ with $\pl \Phi^1 = J^1$.

By Lemma \ref{lem:extcol}.2,
for each clique $Q \in K^r_q(K^r_n)$ we can choose some $\phi(\OO)$ with $\phi(\hQ^+)=Q$
and distinct colours $(i_e: e \in \OO \sm \hQ^+)$ such that each $\phi(e) \in \sS_{i_e}(K^*)$.
Moreover, by the same proof, fixing some $i_{e'}$ with $e' \in \sS_{i_{e'}}(K^*)$
for each $e' \in Q \cap \bigcup_{i=1}^u \sS_i(K^*)$,
we can choose $(i_e: e \in \OO \sm \hQ)$ distinct from all such $i_{e'}$.
Adding $\pm(\phi(\Ups^-) - \phi(\Ups^+))$ for each such signed element $\pm Q$ of $\Phi^1$ 
we replace $\Phi^1$ with  $\Phi^2$ with $\pl \Phi^2 = \pl \Phi^1 = J^1$
such that $\Phi^2$ is supported on cliques that are rainbow,
except for at most one edge not in $\bigcup_{i=1}^u \sS_i(K^*)$.

As $J^1_e=0$ for any $e \notin \bigcup_{i=1}^u \sS_i(K^*)$,
we can partition the signed cliques in $\Phi^2$ using $e$ 
into cancelling pairs $Q^\pm$ of opposite sign, 
each of which uses $e$ and has all other edges in $\bigcup_{i=1}^u \sS_i(K^*)$.
For each such $Q^\pm$, by Lemma \ref{lem:extcol}.1
we can fix some rainbow clique $Q'$ that intersects each of $Q^\pm$
precisely in $e$ and has all  other edges in $\bigcup_{i=1}^u \sS_i(K^*)$.
Then by Lemma \ref{lem:extcol}.3, we can find elimination configurations
for $(Q^+,Q')$ and $(Q',Q^-)$, which produce sets of rainbow cliques
equivalent to $Q^+ - Q'$ and $Q' - Q^-$. Subtracting all such configurations,
we eliminate all cancelling pairs $Q^\pm$ and
replace $\Phi^2$ with  $\Phi^3$ with $\pl \Phi^3 = \pl \Phi^2 = J^1$
such that $\Phi^3$ is supported on rainbow cliques.

Next, by Lemma  \ref{lem:extcol}.4 we can replace
$\Phi^3$ with  $\Phi^4$ with $\pl \Phi^4 = \pl \Phi^3 = J^1$ such that 
$\Phi^4$ is supported on monochromatic unsaturated cliques 
$\sS_i(Q)$ with $i \in [u]$ and $Q \in K^r_q(K) \sm \mc{S}$. 
By Lemma \ref{lem:KSG}.iii each such $\pl Q \in \bgen{\mc{G}}_\GG$.
Recalling that $\mc{Q}_0 := \bigcup_{i=1}^u \sS_i(\mc{G})$,
writing $\hat{J}^1 \in \GG^{K^r_n}$ for the image of $J^1$
under the mod $N$ projection $\mb{Z} \to \GG$ on each edge,
we can replace  $\Phi^4 \in \mb{Z}^{K^r_q(K^r_n)}$ 
by $\Phi^5 \in \GG^{\mc{Q}_0}$ with $\pl \Phi^5 = \hat{J}^1$.

Obtain $\Phi^6 \in [N]^{\mc{Q}_0}$ from $\Phi^5$
by replacing each entry by its representative in $[N]$.
Then $J^2 := J^1 - \pl \Phi^6 \in N \mb{Z}^{\bigcup \mc{Q}_0}$.
By Lemma \ref{lem:Q0'} each $e \in \bigcup \mc{Q}_0$ is in some
`locally decoding'  $(q+r)$-set $Z_e$ with $K^r_q( \tbinom{Z_e}{r} ) \sub \mc{Q}'_0$,
so $J^2 = \pl \Phi^7$ with $\Phi^7 \in \mb{Z}^{\mc{Q}'_0}$.
To conclude, $\Phi \in \mb{Z}^{\mc{Q}'_0}$ 
defined by $\Phi = \pl \Phi^0 + \Phi^6 + \Phi^7$
has $\pl \Phi = J$, as required.
\end{proof}

The remainder of the paper is essentially expository,
with some simplifications and quantitative improvements.

\section{Local decoders and regularity boosting}  \label{sec:decode+boost}

We start this section with a short proof of the local decoder lemma, following Wilson \cite{W}.
 
\begin{proof}[Proof of Lemma \ref{lem:decode}] 
Write $N = r!k = (q)_r := q(q-1) \dots (q-r+1)$.
The existence of $\Psi \in \mb{Z}^{K^r_q(K^r_{r+q})}$ with $\pl \Psi = N \cdot 1_e$
is equivalent to the existence of a solution $x \in \mb{Z}^{\tbinom{[q+r]}{q}}$ 
to the equation $Mx=v$, where $v = N \cdot 1_e \in  \mb{Z}^{\tbinom{[q+r]}{r}}$ 
and $M$ is the inclusion matrix with entries 
$M_{e',Q} = 1_{e' \sub Q}$ for $e' \in  \tbinom{[q+r]}{r}$, $Q \in \tbinom{[q+r]}{q}$.
If such $x$ exists, then for any $I \sub [q+r]$ with $|I|=i \le r$, 
summing the equations for $e' \in  \tbinom{[q+r]}{r}$ containing $I$ gives
$N \cdot 1_{I \sub e} = \tbinom{q-i}{r-i} \sum_{Q \sups I} x_Q$,
so $\sum_{Q \sups I} x_Q = (q)_i (r-i)! 1_{I \sub e} \in \mb{Z}$. 
We can then solve for $x$ using the inclusion-exclusion formula,
and the resulting formula shows that a (unique) solution
$x$ indeed exists by reversing the steps of the calculation:
\[ x_{Q'} = \sum_{I \sub [q+r] \sm Q'} (-1)^{|I|} \sum_{Q \sups I} x_Q
=  \sum_{I \sub [q+r] \sm Q'} (-1)^{|I|}  (q)_{|I|} (r-|I|)! 1_{I \sub e} \in \mb{Z}.\] 
Writing $t = |e \sm Q'|$, we have $x_{Q'}=\sum_{i=0}^t (-1)^i \tbinom{t}{i} (q)_i (r-i)!$,
so $|x_{Q'}| \le \sum_i \tbinom{r}{i} (q)_i (r-i)! = 2^q r!$.
\end{proof}

\begin{rem} \label{rem:div}
The above proof also characterises $K^r_q$-divisible $J \in \mb{Z}^{K^r_{r+q}}$,
that is, $J$ of the form $\pl \Psi$ for some $\Psi \in \mb{Z}^{K^r_q(K^r_{r+q})}$,
by the divisibility conditions $\tbinom{q-i}{r-i} \mid |J(I)|$ for all $0 \le i \le r$, $I \in K^i_{r+q}$.
Indeed, consider $J$ satisfying these divisibility conditions, 
and write $J=\pl\Psi$ with $\Psi := M^{-1} J \in  \mb{Q}^{K^r_q(K^r_{r+q})}$.
Then for each  $I \sub [q+r]$ with $|I|=i \le r$, summing the equations $J_e = \pl \Phi_e$
for $e \sups I$ gives $|J(I)| = \tbinom{q-i}{r-i}  \sum_{Q \sups I} \Psi_Q$,
so $\sum_{Q \sups I} \Psi_Q \in \mb{Z}$ by assumption on $J$.
By inclusion-exclusion we deduce $\Psi_Q \in \mb{Z}$ for all $Q$.
One can then show (still following \cite{W}) by induction that the divisibility conditions
characterise $K^r_q$-divisibility in $K^r_n$ for all $n \ge r+q$.
To do so, consider any $J \in \mb{Z}^{K^r_n}$ with $n>r+q$ satisfying the divisibility conditions.
Then $J(n) \in \mb{Z}^{K^{r-1}_{n-1}}$ satisfies the divisibility conditions for $K^{r-1}_{q-1}$,
so by induction $J(n) = \pl \Psi'$ for some $\Psi' \in \mb{Z}^{K^{r-1}_{q-1}(K^{r-1}_{n-1})}$.
We define $\Psi \in  \mb{Z}^{K^r_q(K^r_n)}$ by $\Psi_{Q' \cup \{n\}} = \Psi_{Q'}$.
Then $J - \pl \Psi$ is supported in $J \in \mb{Z}^{K^r_{n-1}}$ and satisfies the divisibility conditions,
so has the required expression by induction.
\end{rem}

Besides the absorber construction,
we also use local decoders for the boost lemma,
following \cite[Lemma 6.3]{GKLO}.

\begin{proof}[Proof of Lemma \ref{lem:reg}] 
Let $G \sub K^r_n$ where $n \ge n_0$ and $K^r_n \sm G$ is $c$-bounded, with $c<2^{-3q}$.
We will obtain $H \sub K^r_q(G)$ by selecting cliques $Q$
independently with some probabilities $(p_Q)$
satisfying $\sum_{Q \sups e} p_Q = 1/2$ for all $e \in E(G)$.
This is sufficient to prove the lemma, as then
whp each $|H(e)| = (1/2 \pm n^{-1/3}) \tbinom{n}{q-r}$ by Chernoff bounds.

To construct $(p_Q)$, we start with the approximately correct assignment
$p'_Q = \tfrac12 \tbinom{n}{q-r}^{-1}$ for all $Q \in K^r_q(G)$, 
to which we then make a small adjustment using the local decoders.
As $K^r_n \sm G$ is $c$-bounded, each $e \in G$ is contained
in $(1 \pm 2kc) \tbinom{n}{q-r}$ cliques in $K^r_q(G)$.
Thus we can write $\sum_{Q \sups e} p'_Q = \tfrac12 - c_e$ 
with each $|c_e| < 3kc$.
Similarly, each $e \in G$ is contained
in $(1 \pm 2\tbinom{q+r}{r}c) \tbinom{n}{q}$ cliques in $K^r_{q+r}(G)$.
For each $e \in G$, we fix an arbitrary family $\mc{Z}_e$ of $\tfrac12 \tbinom{n}{q}$ such cliques.
For each $Z \in \mc{Z}_e$, we let $\Psi^Z$ be the copy of $\Psi$ as in Lemma \ref{lem:decode}
with $\pl \Psi^Z = N \cdot 1_e$ and each $|\Psi_Q| \le 2^q r!$.

We obtain $(p_Q)$ from $(p'_Q)$ by adding $c_e (N|\mc{Z}_e|)^{-1}  \Psi^Z$
for each $e \in G$ and $Z \in \mc{Z}_e$, so that each $p_Q = p'_Q + p''_Q$
with $p''_Q = \sum_e \sum_{Q \sub Z \in \mc{Z}_e} c_e (N|\mc{Z}_e|)^{-1} \Psi^Z_Q$.
For each $e,e' \in G$ and $Z \in \mc{Z}_e$ we have
$\sum_{e' \sub Q \sub Z} \Psi^Z_Q = \pl \Psi^Z_{e'} = N 1_{e=e'}$,
so for each $e' \in G$ we have $\sum_{Q \sups e'} p''_Q
= \sum_e \sum_{Z \in \mc{Z}_e} N 1_{e=e'} \cdot   c_e (N|\mc{Z}_e|)^{-1} = c_{e'}$,
so $\sum_{Q \sups e'} p_Q = \tfrac12 - c_{e'} + c_{e'} = \tfrac12$.

It remains to show that $p_Q \in [0,1]$. To see this we bound $p''_Q$,
noting that for each $Q$ there are at most $\tbinom{q+r}{r} \tbinom{n}{r}$
choices of $(e,Z)$ with $Q \sub Z \in \mc{Z}_e$,
for which the summand $c_e (N|\mc{Z}_e|)^{-1} \Psi^Z_Q$
is bounded by $3kc   (N \tfrac12 \tbinom{n}{q})^{-1} 2^q r!$,
so $|p''_Q| \le \tbinom{q+r}{r} \tbinom{n}{r} \cdot 6c  \tbinom{n}{q}^{-1} 2^q < \tfrac12 \tbinom{n}{q-r} = p'_Q$
as $c < 2^{-3q}$. Thus $p_Q \in [0,1]$, as required.
\end{proof}

\section{Martingale concentration} \label{sec:conc}

Here we discuss a martingale concentration inequality of Freedman,
which implies the concentration inequalities used earlier,
and will also be used to prove the nibble lemma in the next section.
We make the following standard definition,
which we only use for the natural filtration
$\mc{F}=(\mc{F}_i)_{i \ge 0}$ associated with a random process, 
where each $\mc{F}_i$ consists of all events 
determined by the history of the process up to step $i$.

\begin{defn}
Let $\OO$ be a (finite) probability space.
An \emph{algebra} (on $\OO$) is a set $\mc{F}$ of subsets of $\OO$ 
that includes $\OO$ and is closed under intersections 
and taking complements.
A \emph{filtration} (on $\OO$) is a sequence 
$\mc{F}=(\mc{F}_i)_{i \ge 0}$ 
of algebras such that $\mc{F}_i \sub \mc{F}_{i+1}$ for $i \ge 0$.
A sequence $A = (A_i)_{i \ge 0}$ of random variables on $\OO$
is a \emph{martingale} (for $\mc{F}$)
if each $A_i$ is {\em $\mc{F}_i$-measurable}
(all $\{\oO: A_i(\oO)<t\} \in \mc{F}_i$) 
and $\mb{E}(A_{i+1}|\mc{F}_i) = A_i$ for $i \ge 0$.
\end{defn}

Now we can state a general result of Freedman \cite[Proposition 2.1]{F}.

\begin{lemma} \label{freed} 
Let $(A_i)_{i=0}^t$ be a martingale for $\mc{F}=(\mc{F}_i)_{i \ge 0}$
with each $|A_{i+1}-A_i| \le b$. 
Let $E$ be the `bad' event that there 
exists $j > 0$ with $A_j \ge A_0 + a$ 
and $\sum_{i<j} \var[A_{i+1} \mid \mc{F}_i] \le v$.
Then $\mb{P}(E) \le \exp \brac{-\tfrac{a^2}{2(v+ab)}}$.
\end{lemma}

The standard technique for proving such concentration inequalities is bounding the moment generating function
and applying Markov's inequality; we include a short proof for completeness.

\begin{proof}
We can assume $b=1$, by replacing $(A_i,a,v)$ by $(A_i,a/b,v/b^2)$.
We consider $f(z) = (e^z-z-1)/z^2$, 
noting that $f'(z) = 2z^{-3}(e^z-1)+z^{-2}(e^z+1)>0$,
so for $t \in [-1,1]$ and $z>0$
we have $f(zt) \le f(z)$, so $e^{zt} \le 1+zt+(e^z-z-1)t^2$.
We set $t = A_{i+1} - A_i  \in [-1,1]$ 
and apply $\mb{E}^i := \mb{E}[ \cdot \vert \mc{F}_i ]$.
This gives
$\mb{E}^i e^{z(A_{i+1} - A_i)} \le 1 + (e^z-z-1)  \mb{E}^i  (A_{i+1} - A_i)^2
\le V_i :=  \exp [ (e^z-z-1)  \mb{E}^i  (A_{i+1} - A_i)^2 ]$,
so $M_j := e^{z(A_j-A_0)} \prod_{i<j} V_i^{-1}$
satisfies  $\mb{E}^i M_{i+1} \le M_i$.
We consider the stopping time $\tau$ when first $A_\tau \ge A_0 + a$,
or $\tau=t$ if $E$ does not hold. By the Optional Stopping Theorem,
$1 \ge \mb{E}[M_\tau] \ge \mb{P}(E) e^{za} e^{-(e^z-z-1)v}$.
We set $z=\log(1+a/v) = -\log(1-c) \ge c+c^2/2$, where $c:=a/(a+v)$,
giving $e^z v = a+v$ and $z(a+v) \ge  a + a^2/2(a+v)$,
so $\mb{P}(E) \le e^{(e^z-z-1)v - za} \le e^{-a^2/2(a+v)^2}$.
\end{proof}

\begin{rem} \label{rem:doob}
Lemma \ref{freed} also applies under the assumption that $(A_i)_{i=0}^t$ is a supermartingale,
meaning that each $\mb{E}(A_{i+1}|\mc{F}_i) \le A_i$. To see this, consider the Doob decomposition
$A_i = M_i + P_i$, where $M_0=A_0$, $P_0=0$, and for $i \ge 0$ we let
$M_{i+1} - M_i = A_{i+1} - \mb{E}(A_{i+1}|\mc{F}_i)$ and $P_{i+1} - P_i = \mb{E}(A_{i+1}|\mc{F}_i) - A_i$,
so that $M_i$ is a martingale and $P_i$ is a decreasing predictable process. 
Then $(M_i)_{i=0}^t$ is a martingale with each $|M_{i+1}-M_i| \le b$
and $\var[M_{i+1} \mid \mc{F}_i] = \var[A_{i+1} \mid \mc{F}_i]$,
and the conclusion of Lemma \ref{freed} for $M_i$ implies that for $A_i$.
\end{rem}

\begin{proof}[Proof of Lemma \ref{lem:pseudobin}.]
Let $X = \sum_{i=1}^n X_i$ where $X_i$ are random variables with each $|X_i| \le C$.
We let $\mc{F}$ be the natural filtration.
For (ii), we assume $\sum_i |X'_i| \le \mu$, where $X'_i := \mb{E}[|X_i| \mid \mc{F}_{i-1}]$.
Considering the martingale $A_i = \sum_{j \le i} (|X_i| - X'_i)$,
we have $\var[A_i \mid \mc{F}_{i-1}] \le \mb{E}[ X_i^2  \mid \mc{F}_{i-1}] \le CX'_i$,
so by Lemma \ref{freed} applied with $a=c\mu$, $b=2C$ and $v=C\mu$
we obtain $\mb{P}(|X|>(1+c)\mu) \le 2e^{-\mu c^2/2(1+2c)C}$.
For (i), we assume the $X_i$ are independent and $\mb{E}X=\mu$.
Then a similar calculation for $A_i = \sum_{j \le i} (X_i - \mb{E}X_i)$ and $-A_i$
gives $\mb{P}(|X-\mu|>c\mu) \le 2e^{-\mu c^2/2(1+2c)C}$.
\end{proof}

\section{The clique removal process} \label{sec:removal}

We concluding by proving the nibble lemma, 
which we state below in a more general form: for Lemma \ref{lem:nibble}
we substitute $\eps=1/3$, $\tT=\phi=1/2$ and obtain $L_1 := G \sm \bigcup D_1$
that is  $n^{-3k\rho}$-bounded as $\rho = (6k)^{-2}$.
For the proof we apply the clique removal process
(following the analysis of Bennett and Bohman \cite{BB}),
which provides a much more precise analysis of random greedy algorithms
than that given in Section \ref{sec:rga}, leading to better bounds than standard nibble arguments,
which only give boundedness $n^{-O(\eps/k^2)}$.

\begin{lemma} \label{lem:nibble+} 
Suppose $G \sub K^r_n$ and $H \sub K^r_q(G)$ with $|G| = \phi \tbinom{n}{r}$
and $|H(e)| =  (1 \pm n^{-\eps}) \tT \tbinom{n}{q-r}$ for all $e \in V(H) = E(G)$,
where $\eps < 0.4$, $\tT \ge n^{-1/3}$, $\phi \ge n^{-r/3}$ and $n>n_0$ as in \eqref{param}. 
Then there is a set $H' \sub H$ of edge-disjoint $q$-cliques
such that $G \sm \bigcup H'$ is $3n^{-\eps/3k}$-bounded, where $k:=\tbinom{q}{r}$.
\end{lemma} 

The density assumption on $G$ is somewhat redundant,
as the assumption on $|H(e)|$ implies a lower bound on $\phi$ by Kruskal-Katona, 
but we omit this argument as we know anyway that $G$ is dense in our intended application.

\begin{proof}
We construct $H'$ by the clique removal process, 
where in each step we choose a uniformly random $q$-clique of $H$ 
among those edge-disjoint from all previous choices.
The $r$-graph $G^i$ after $i$ steps satisfies
$|G^i|=|G|-ki = p|G|$ with $k=\tbinom{q}{r}$, $p=1-kt$ and $t=i/|G|$.
We will see that it whp resembles a $p$-random subgraph of $G$.
We write $H(i) = H \cap K^r_q(G^i)$ and $H_e(i) = \{Q: e \in Q \in H(i) \}$ for any $e \in G^i$.
By assumption $|H_e(0)|=|H(e)| =  (1 \pm n^{-\eps})D$ with $D := \tT \tbinom{n}{q-r}$,
so $k|H(0)| = \sum_{e \in G} |H_e(0)| =  (1 \pm n^{-\eps}) D |G|$.
 
We will show whp for $0 \le i \le i^* := (1-n^{-\eps/3k})|G|/k$ that
$|H(i)| = p^k D|G|/k \pm e_H$ and $|H_e(i)| = p^{k-1}D \pm e_D$, 
\[ \text{ where } \ \ e_H = 6(1-k\log p)^2 bD|G|, \quad e_D = 2(1-k\log p) b^c D,
\quad b :=n^{-\eps}, \quad c:=2/3. \]
(We could substitute the values of $b,c$ throughout, 
but will keep the notation general to clarify how they were chosen.
The stated estimates for $|H(0)|=|H|$ and $|H_e(0)|=|H(e)|$ 
hold as $e_H(0) = 6n^{-\eps}D|G|$ and $e_D(0) = 2n^{-2\eps/3} D$.)

We just show the upper bounds, as the lower bounds are similar.
Equivalently, we need $H^+(i) \le 0$ and $H^+_e(i) \le 0$,
where $H^+(i) := |H(i)| - p^k D|G|/k - e_H$ 
and $H^+_e(i) := |H_e(i)| - p^{k-1}D - e_D$.

We will verify the following hypotheses, 
in which $\mb{E}'$ denotes conditional expectation
given the history of the process, 
and we can assume that these bounds hold at earlier steps
by considering $i<\tau$, where the stopping time $\tau$
is the first step where the bounds fail, or $\infty$ if there is no such step.
We also assume $i \le i^*$, so $p^k \ge b^{1/3} = n^{-\eps/3}$.

\emph{Trend:} if $H^+(i) \ge -bD|G|$ then $\mb{E}' H^+(i+1) \le H^+(i)$;
  if $H^+_e(i) \ge -b^c D$ then $\mb{E}' H^+_e(i+1) \le H^+_e(i)$.
  
\emph{Boundedness:}  $|H^+(i+1)-H^+(i)| \le 2kD$
 and $|H^+_e(i+1)-H^+_e(i)| \le 2qn^{q-r-1}$. 
 
To see boundedness, note that we select some clique $Q$, as $i<\tau$ we have 
$||H(i+1)|-|H(i)|| \le \sum_{e \in Q} |H_e(i)| \le k \max_e |H_e| <1.1kD$
and each $||H_e(i+1)|-|H_e(i)|| \le \sum_{e' \in Q} |H_{e'}(i) \cap H_e(i)| \le 1.1qn^{q-r-1}$.
(The much smaller changes in the deterministic functions $H^+(i)-|H(i)|$ and $H_e^+(i)-|H_e(i)|$
are covered by changing $1.1$ to $2$.)

Next we assume the trend hypothesis and apply Lemma \ref{freed}
(actually Remark \ref{rem:doob}) to show that whp $\tau=\infty$.

For $H^+(i)$, we have $|H^+(i+1)-H^+(i)| = O(D)$
and $\sum_i \var[H^+(i+1)\mid \mc{F}_i] = \sum_i O(D^2) = O(|G|D^2)$,
so $\mb{P}(H^+(i)-H^+(0) \ge bD|G|)$ has exponential decay as $b^2 = n^{-2\eps} \gg \phi^{-1} n^{-r}$. 
For each $H^+_e(i)$, we have $|H^+_e(i+1)-H^+_e(i)| = O(n^{q-r-1})$
and  $\sum_i \var[H^+_e(i+1) \mid \mc{F}_i] 
\le \max_i |H^+_e(i+1)-H^+_e(i)| \sum_i \mb{E}[  |H^+_e(i+1)-H^+_e(i)|  \mid \mc{F}_i] 
\le O(n^{q-r-1} D)$, so $\mb{P}(H^+(i)-H^+(0) \ge b^c D)$ has exponential decay 
as $b^{2c} = b^{4/3} = n^{-4\eps/3} \gg (\tT n)^{-1}$, recalling  $\eps < 0.4$ and $\tT \ge n^{-1/3}$.

To verify the trend hypothesis, we consider the estimate
\[ \mb{E}'[ |H(i+1)|-|H(i)| ] =  - \sum_{e \in G^i} |H_e(i)|^2/|H(i)| + O(n^{q-r-1}), \]
which holds as for each edge $e \in G^i$ we select a clique containing $e$
with probability $|H_e(i)|/|H(i)|$ and then remove $|H_e(i)|$ cliques containing $e$;
the error term accounts for cliques that share more than one edge with the selected clique.
As $i<\tau$, each $(|H_e(i)| - p^{k-1}D)^2 \le e_D^2$.
Consider the random variable $X=|H_e(i)|$ for uniformly random $e \in G^i$.
Then $e_D^2 \ge \var(X) = |G^i|^{-1} \sum_e |H_e(i)|^2 - (|G^i|^{-1} \sum_e |H_e(i)|)^2$,
giving $\sum_e |H_e(i)|^2 = |G^i|^{-1} (k |H(i)|)^2 \pm |G^i| e_D^2$, so
\[ \mb{E}'[ |H(i+1)|-|H(i)| ] = - k^2|H(i)|/|G^i| \pm   |G^i| e_D^2/|H(i)| + O(n^{q-r-1}). \]
Assuming $H^+(i) \ge -bD|G|$, that is $|H(i)| \ge p^k D|G|/k + e_H - bd|G|$, 
and recalling $|G^i|=p|G|$, we obtain 
\[ \mb{E}'[ |H(i+1)|-|H(i)| ]
\le -kDp^{k-1} - k^2 e_H / p|G| + k^2 bD/p + (k+o(1)p^{1-k} e_D^2/D,\]
where in the last term we use $k|H(i)|/|G^i| \sim p^{k-1} D$ 
and $p^{1-k} e_D^2/D \gg  n^{q-r-1}$ as $b^{2c} = b^{4/3} \gg (\tT n)^{-1}$. 
As $e_H$ is increasing
and $p(t+|G|^{-1})^k - p(t)^k = (1-kt-k|G|^{-1})^k - (1-kt)^k = -(1+O(|G|^{-1}))p^{k-1} k^2|G|^{-1}$, 
\begin{align*} 
\mb{E}'[ H^+(i+1)-H^+(i) ] & \le \mb{E}'[ |H(i+1)|-|H(i)| ] + (1+O(|G|^{-1}))p^{k-1}kD \\
& \le - k^2 e_H / p|G| + k^2 bD/p + (k+1)p^{1-k} e_D^2/D \le 0,
\end{align*}
using $|G|^{-1} p^{k-1}kD \ll p^{1-k} e_D^2/D$ 
and $k^2 p^{-1}(e_H/|G|-bD) = k^2 p^{-1}( 6(1-k\log p)^2-1) bD 
\ge (k+1)p^{1-k} e_D^2/D = (k+1) p^{1-k} \cdot 4(1-k\log p)^2 b^{2c} D$,
as $b^{2c-1} = b^{1/3} \le p^k$. 

This verifies the trend hypothesis for $H^+(i)$. For $H^+_e(i)$ with $e \in G^i$ we consider
\[ \mb{E}'[ |H_e(i+1)|-|H_e(i)| ] =  - |H(i)|^{-1} \sum_{Q \in H_e(i)} \sum_{e' \in Q \sm \{e\}} (|H_{e'}(i)| + O(n^{q-r-1})), \]
which holds as for each clique $Q \in H_e(i)$ and $e' \in Q \sm \{e\}$ we select a clique containing $e'$
and so remove $Q$ with probability $|H_{e'}(i)|/|H(i)|$; the error term accounts for $Q \in H_e(i)$ 
sharing more than one edge with the selected clique.
Assuming $H^+_e(i) \ge -b^c D$, that is $|H_e(i)| \ge p^{k-1} D + e_D - b^c D$, we obtain
\begin{align*} 
& \mb{E}'[ |H_e(i+1)|-|H_e(i)| ]  \le  - |H(i)|^{-1} ( p^{k-1} D + e_D - b^c D) (k-1) ( p^{k-1} D - e_D - O(n^{q-r-1}) ) \\
& \le k(k-1) p^{k-2} D |G|^{-1} + (1+o(1)) b^c D  \cdot (k-1)  p^{k-1} D \cdot (p^k D|G|/k)^{-1} + O(e_H/p^2|G|^2),
\end{align*}
using $|H(i)| = p^k D|G|/k \pm e_H$ as $i<\tau$, 
and $e_D \ll p^{k-1} D$ as $b^c = b^{2/3} \le p^k$, 
and $(q-r)! D^{-1} n^{q-r-1} \sim (\tT n)^{-1} \ll b^c = o(1)$.
As $e_D(t+|G|^{-1}) - e_D(t) = (1+O(|G|^{-1}))|G|^{-1}e_D'(t) =  (1+O(|G|^{-1}))|G|^{-1} \cdot 2k^2 b^c D/p$
and $p(t+|G|^{-1})^{k-1} - p(t)^{k-1} = -(1+O(|G|^{-1}))p^{k-2} k(k-1)|G|^{-1}$, we deduce
\begin{align*} 
\mb{E}'[ H^+_e(i+1)-H^+_e(i) ] & \le \mb{E}'[ |H_e(i+1)|-|H_e(i)| ] + (1+O(|G|^{-1}))|G|^{-1}D (  p^{k-2} k(k-1) -  2k^2 b^c /p)   \\
& \le   -k^2 b^c D/ p |G| + O(e_H/p^2|G|^2) + O(Dp^{k-2}/|G|^2) \le 0,
\end{align*}
using $p^k b^{-c} \ll |G|$ and $b^{1-c} = b^{1/3} \le p^2$. 
This establishes the trend hypothesis for $H^+_e(i)$.

Now we show that whp $G \sm \bigcup H'$ is $3n^{-\eps/3k}$-bounded.
We show whp for each $f \in \tbinom{[n]}{r-1}$ that $|G^i(f)| = p|G(f)| \pm e_F$, where $e_F := 2b^{c/2} n$.
For the upper bound $G^i(f)^+ := |G^i(f)| - p|G(f)| - e_F \le 0$ we verify the following hypotheses.

\emph{Trend:} if $G^i(f)^+ \ge -b^{c/2} n$ then $\mb{E}' G^{i+1}(f)^+ \le G^i(f)^+$.

\emph{Boundedness:}  $|G^{i+1}(f)^+ - G^i(f)^+| \le 2q$.

Boundedness is clear, as removing a clique $Q$ only affects $G^i(f)$ if $f \sub Q$,
in which case it removes $q-r+1$ edges containing $f$. Assuming the trend hypothesis,
the required bound follows whp from Lemma \ref{freed}, as 
$\sum_i \var[G^{i+1}(f)^+\mid \mc{F}_i] 
\le \max_i |G^{i+1}(f)^+ - G^i(f)^+| \sum_i \mb{E}[  |G^{i+1}(f)^+ - G^i(f)^+|  \mid \mc{F}_i] 
\le O(n)$, so $\mb{P}(H^+(i)-H^+(0) \ge b^{c/2} n)$ has exponential decay 
as $b^c = b^{2/3} \gg n^{-1}$. 
For the trend hypothesis, if $G^i(f)^+ \ge -b^{c/2} n$ then
\begin{align*} 
& \mb{E}'[ |G^{i+1}(f)|-|G^i(f)|]  = - \sum_{e: f \sub e \in G^i} |H_e(i)|/|H(i)| \\
& \le -( p|G(f)| + e_F - b^{c/2} n ) ( p^{k-1}D - e_D ) / ( p^k D|G|/k + e_H ) \\
& \le  - k |G(f)|/|G| - k(e_F - b^{c/2} n)/p|G| 
+ O(e_D) \tfrac{|G(f)|}{  p^{k-1} D|G| } + O(e_H) \tfrac{|G(f)|}{  p^k D|G|^2 }. 
\end{align*}
The one-step change in $p|G(f)|+e_F$ is $-k|G|^{-1}|G(f)|$, so
\begin{align*} 
& \mb{E}'[ G^{i+1}(f)^+ - G^i(f)^+] =  \mb{E}'[ |G^{i+1}(f)|-|G^i(f)|]  + k|G|^{-1}|G(f)| \\
& \le  - kb^{c/2}  n/p|G| + O(e_D p + e_H |G|^{-1}  ) n (p^k D|G|)^{-1}  \le 0,
\end{align*}
using $pe_H \ll b^{c/2} p^k D |G|$ as $b^{1-c/2} = b^{2/3} \le p^k$ 
and $p^2 e_D \ll b^{c/2} p^k D$ as $b^{c/2} = b^{1/3} \le p^k$.

At the end, for each $f \in \tbinom{[n]}{r-1}$ we have
$|(G \sm \bigcup H')(f)| = |G^{i^*}(f)| = n^{-\eps/3k} |G(f)| \pm e_F \le 3n^{1-\eps/3k}$.
\end{proof}

\section{Quantitative considerations} \label{sec:quant}

To avoid clutter, we have not explicitly kept track of the minimum value of $n$ 
needed for various inequalities throughout the paper, as our main purpose was
to present a new short self-contained proof of the existence of designs.
However, it is also interesting to quantify the dependence of $n$ on $q$ and $r$.
The bottleneck for our choice of $n_0$ in \eqref{param} 
was the size of the clique exchange configuration, 
for which we achieved $|\OO| \le 3(2q)^r k^2$ in Lemma \ref{lem:OO}.
The tightest inequality on $n_0$ used in the proof was $n_0 > (2^{5q} (4q)^r u)^{10/\aA}$
in Lemma \ref{lem:flat}, which holds as $\aA^{-1} =  (2q)^r (6k)^2$,
$u = 20q^2 \aA^{-1} |\OO|$ and $n_0 = (4q)^{90q/\aA}$.
We need to take  $\aA^{-1}$ larger than $\OO$ due to density factors
$n^{-\aA|\OO|}$ appearing in the proof of Lemma \ref{lem:extcol}.
The  constraint on $\aA$ from the nibble is less restrictive:
we take $\rho = (6k)^{-2}$ so that $L_1$ in Lemma \ref{lem:nibble} is $n^{-3k\rho}$-bounded,
and then $k\aA < \rho/2$ is small enough for Lemma \ref{lem:Q0'}.
The bounds on $n$ required for concentration inequalities are much less restrictive,
and the tightest constraint on $n$ for the random greedy algorithms comes from
the assumption $\tT > n^{-1/2}$ in Lemma \ref{lem:process}.
There are several other places where we need $n$ to grow exponentially in $q/\aA$,
as should be expected, as the boundedness of the absorber $A$ in Lemma \ref{lem:A} is $n^{-\aA/4}$,
which we bound by $2^{-3q}$ to apply Lemma \ref{lem:reg}. Thus we obtain 
$\log n_0 = O(\aA^{-1}q \log q) = O(k^2 q^{r+1}\log q)$, whereas it would be
interesting to show that one can take $\log n_0 = O(k)$ where $k=\tbinom{q}{r}$, 
as this is a natural barrier for any probabilistic construction.

\end{document}